\documentclass[12pt,twoside,a4paper]{amsart}
\usepackage[utf8]{inputenc}

\usepackage[inner=2.5cm,outer=3cm,top=2.5cm,bottom=4cm,includeheadfoot]{geometry}
 
\newcommand{\folder}{./}

\usepackage{amsmath}
\usepackage{amssymb}
\usepackage{amsopn}
\usepackage{amsthm}
\usepackage{amsfonts}
\usepackage[right]{eurosym}
\usepackage{mathtools}

\usepackage{color} 
\usepackage{graphicx}

\usepackage{tikz}
\usepackage[all]{xy}

\usepackage{longtable}
\usepackage{totpages}

\usepackage[utf8]{inputenc} 
\usepackage[T1]{fontenc}

\usepackage{mathrsfs}
\usepackage{paralist}

\usepackage{longtable}

\definecolor{darkblue}{rgb}{0,0,0.6}
\usepackage[ocgcolorlinks,colorlinks=true, citecolor=darkblue, filecolor=darkblue, linkcolor=darkblue, urlcolor=darkblue]{hyperref}

\usepackage{tikz}
\usetikzlibrary{matrix,arrows,calc,snakes,patterns,decorations.markings}
\usepackage{tikz-cd}

\usepackage{tkz-euclide}
\tikzset{
  dotted/.style={pattern=dots,pattern color=#1},
  dotted/.default=black
}
\tikzset{
  fdotted/.style={pattern=crosshatch dots,pattern color=#1},
  fdotted/.default=black
}
\tikzset{
  scopedlines/.style={pattern=north east lines,pattern color=#1},
  scopedlines/.default=black
}
\tikzset{
  hrlines/.style={pattern=horizontal lines,pattern color=#1},
  hrlines/.default=black
}

\usepackage{eucal}
\usepackage{wasysym}

\usepackage{xspace} 

\DeclareMathOperator{\rQ}{Q}

\newcommand{\bC}{{\mathbb C}}

\newcommand{\bN}{{\mathbb N}}

\newcommand{\bP}{{\mathbb P}}
\newcommand{\bQ}{{\mathbb Q}}
\newcommand{\bR}{{\mathbb R}}

\newcommand{\bZ}{{\mathbb Z}}

\newcommand{\cA}{{\mathscr A}}
\newcommand{\cB}{{\mathscr B}}
\newcommand{\cC}{{\mathscr C}}
\newcommand{\cD}{{\mathscr D}}

\newcommand{\cK}{{\mathscr K}}

\newcommand{\cM}{{\mathscr M}}

\newcommand{\cX}{{\mathscr X}}

\newcommand{\dO}{{\mathcal O}}
\newcommand{\dP}{{\mathcal P}}

\newcommand{\fB}{{\mathfrak B}}

\newcommand{\fI}{{\mathfrak I}}

\newcommand{\eps}{\varepsilon}
\renewcommand{\phi}{\varphi}

\DeclareMathOperator{\Proj}{Proj}

\DeclareMathOperator{\Xh}{\hat{X}}
\DeclareMathOperator{\iso}{\cong}
\newcommand{\twolines}[2]{\genfrac {}{}{0pt}{}{#1}{#2}}

\DeclareMathOperator{\dual}{^{\vee}}					
\DeclareMathOperator{\pr}{pr}
\DeclareMathOperator{\inj}{\hookrightarrow}

\DeclareMathOperator{\Sym}{Sym}
\DeclareMathOperator{\too}{\longrightarrow}
\DeclareMathOperator{\rank}{rk}

\DeclareMathOperator{\Pic}{Pic}

\DeclareMathOperator{\NS}{NS}

\DeclareMathOperator{\birKbar}{\overline{\cB\cK}}

\DeclareMathOperator{\Def}{Def}

\DeclareMathOperator{\Coh}{Coh}
\DeclareMathOperator{\Hilb}{Hilb}

\DeclareMathOperator{\tr}{tr}
\newcommand{\Xtild}{\widetilde{X}}

\newcommand{\Dtild}{\widetilde{D}}
\newcommand{\Atild}{\widetilde{A}}

\newcommand{\Xhat}{\hat{X}}

\newcommand{\Chat}{\hat{C}}
\renewcommand{\div}{{\rm div}}

\DeclareMathOperator{\Nef}{Nef}
\DeclareMathOperator{\Amp}{Amp}

\DeclareMathOperator{\half}{\frac{1}{2}}
\DeclareMathOperator{\LambdaKE}{\Lambda_{\rm K3}} 
\DeclareMathOperator{\LambdaKEn}{\Lambda_{{\rm K3}^{[n]}}}
\DeclareMathOperator{\LambdaKEtwo}{\Lambda_{{\rm K3}^{[2]}}}

\DeclareMathOperator{\Supp}{Supp}
\DeclareMathOperator{\NEbar}{\overline{NE}}

\newcommand\restr[2]{{
  \left.\kern-\nulldelimiterspace 
  #1 
  \vphantom{\big|} 
  \right|_{#2} 
  }}

\newif\ifmyversion
 \myversionfalse

\newcommand{\TODO}[1]{}

\newcommand{\Martin}[1]{}

\newdir^{ (}{{}*!/-5pt/@^{(}}  
\newdir_{ (}{{}*!/-5pt/@_{(}}

\theoremstyle{plain}
\newtheorem{proposition}{Proposition}[section]

\newtheorem{lemma}[proposition]{Lemma}

\newtheorem{corollary}[proposition]{Corollary}

\newtheorem{theorem}[proposition]{Theorem}
\newtheorem*{theorem*}{Theorem}
\newtheorem*{conjecture*}{Conjecture}
\newtheorem*{proposition*}{Proposition}
\newtheorem*{corollary*}{Corollary}

\theoremstyle{definition}
\newtheorem{definition}[proposition]{Definition}

\newtheorem*{notation*}{Notation}

\newtheorem{remark}[proposition]{Remark}

\theoremstyle{remark}

\newtheoremstyle{name}
   {}{}{\itshape}{}{\bfseries }{}{ }{\thmname{#3}.}
\theoremstyle{name}
\newtheorem{name}{bla}

\numberwithin{equation}{section}					

\usepackage{fancyhdr}

\usepackage[final]{pdfpages}

\DeclareMathOperator{\SShat}{\widehat{S\times S}}
\DeclareMathOperator{\PPhat}{\widehat{\bP^2\times \bP^2}}
\DeclareMathOperator{\SxS}{S\times S}
\DeclareMathOperator{\PxP}{\bP^2\times \bP^2}
\DeclareMathOperator{\pipi}{\pi\times \pi}
\DeclareMathOperator{\Pdual}{\bP^2\dual}
\DeclareMathOperator{\Pz}{\bP^2}
\DeclareMathOperator{\dOi}{\dO(1)}
\DeclareMathOperator{\dOii}{\dO(1,1)}
\DeclareMathOperator{\pipihat}{\widehat{\pi\times \pi}}
\DeclareMathOperator{\piz}{\pi^{[2]}}

\begin{document}

\title[Non-divisorial base locus of big and nef line bundles on K3$^{[2]}$-type]{On the non-divisorial base
  locus of big and nef line bundles on K3$^{[2]}$-type varieties}  
\author[U. Rie\ss]{Ulrike Rie\ss}
\address{ETH Z\"urich, Institute of theoretical studies, Clausisusstrasse 47, 8092 Z\"urich, Switzerland}
\email{ulrike.riess@eth-its.ethz.ch}

\begin{abstract}
  We approach non-divisorial base loci of big and nef line bundles on irreducible symplectic
  varieties.
  While for K3 surfaces, only divisorial base loci can occur, nothing was known about the
  behaviour of non-divisorial base loci for more general irreducible symplectic varieties.
  
  We determine the base loci of all big and nef line
  bundles on the Hilbert scheme of two points on very general K3 surfaces of genus two and on their birational
  models.
  Remarkably, we find an ample line bundle with a non-trivial base locus in codimension two.
  We deduce that, generically in the moduli spaces of polarized K3$^{[2]}$-type varieties, the
  polarization is base point free.

\end{abstract}

\maketitle

\section{Introduction}

In this article
we study base loci of big and nef line bundles on certain irreducible symplectic varieties.
While in a previous article \cite{Riess18}, we analysed the base divisors of big and nef line bundles, we do a
first approach towards base loci in higher codimension in this article.

Starting point for the analysis of base loci of big and nef line bundles on irreducible symplectic varieties was the observation that K3
surfaces behave special in the context of Fujita's conjecture: Consider an ample line bundle $H$ on a K3
surface $X$. Fujita's conjecture would predict that $3H$ is base point free (see
e.g.~\cite[Conjecture 10.4.1]{LazarsfeldII} for the full statement of Fujita's conjecture). However, it was
shown by 
Mayer (\cite{Mayer72}) that already $2H$ is base point free in this case. 

Remarkably, also for abelian varieties $A$, it is known that  $2H$ is base point free for every ample line
bundle $H\in \Pic(A)$ (see \cite{MumfordAV}).

Note that for abelian varieties and for K3 surfaces, the bounds are even better than predicted by
Fujita's conjecture. This suggests, that it might be particularly interesting to study questions related
to base points of ample line bundles for irreducible symplectic varieties.

\bigskip

The base loci of big and nef line bundles on K3 surfaces are completely understood:

\begin{proposition}[{\cite{Mayer72}}]\label{prop:mayer}
  Let $X$ be a (complex) K3 surface and $H \in \Pic(X)$ a line bundle which is big and nef. 
Then $H$ has base points if and only if $H=\dO(mE+C)$, where $m\geq2$, $E$ is a smooth elliptic curve, and
$C$ is a smooth rational curve, such that $(E, C)=1$. In this case the base locus of $H$ is exactly $C$.
\end{proposition}

Let us mention, that a similar result in arbitrary characteristic $\neq 2$ is
\cite[Proposition 8.1]{Saint-Donat74}.

The proposition has the following immediate consequence:

\begin{corollary}\label{cor:bpfifno0-class}
  If $X$ is a K3 surface with $\Pic(X)=H\cdot \bZ$ for an ample line bundle $H$, then $H$ is base
  point free. In particular for generic elements $(X,H)$ in the moduli space of polarized K3 surfaces, the
  polarization $H$ is base point free. \hfill $\square$ 
\end{corollary}
  
For higher dimensional irreducible symplectic varieties, the only known results are in \cite{Riess18}, where
we discussed the divisorial part of the base loci of big and nef line bundles. Under certain conditions on the
deformation type, which are satisfied for K3$^{[n]}$-type and Kum$^n$-type, and which we expect to hold in general, we were able to give a complete
description.
In particular, we obtain the following analogue of Proposition 1.1 in the case of K3$^{[n]}$-type varieties:

\begin{proposition}[{\cite[Proposition 8.1]{Riess18}}]
  Let $X$ be an irreducible symplectic variety of K3$^{[n]}$-type and $H\in \Pic(X)$ a line bundle that
  is big and nef.  
  Then  $H$ has a fixed divisor if and only if $H$ is of the form
  $H=mL+F$, where $m\geq 2$, $L$ is movable with $q(L)=0$, and $F$ is an irreducible reduced divisor of
  negative square with $(L,F)_q=1$. In this case $F$ is the fixed divisor of $H$.
\end{proposition}

Until now, nothing is known about the behaviour of non-divisorial base loci of big and nef line bundles
on irreducible symplectic varieties of dimension $>2$.
In this article, we do a first step to study these.
We study base points of big and nef line bundle on (four-dimensional) irreducible
symplectic varieties of K3$^{[2]}$-type.

The core of the article is the detailed study of the following example: 
The Hilbert scheme $X\coloneqq\Hilb^2(S)$, where $S$ is a K3 surface with $\Pic(S)\iso \bZ\cdot H_S$ for an ample
line bundle with $(H_S)^2=2$.
The main part of this article is to determine the base loci of all big and nef line bundles on $X$ and on its unique birational model $X'$:

\begin{theorem*}[Compare Theorem \ref{thm:baselocusH+L}, Proposition \ref{prop:mostlyokonX}, and Proposition \ref{prop:okinspeccase}]
  There is a unique big and nef line bundle (which we call $H+L$) on $X$ with non-trivial non-divisorial base
  locus isomorphic to $\bP^2$. All other big and nef line bundles on $X$ are base point free. 
  Furthermore, every big and nef line bundle on $X'$ is base point free.
\end{theorem*}

In particular this gives the first known example of a non-trivial non-divisorial base locus of a big and nef line
bundle on an irreducible symplectic variety.
 This is a phenomenon, which does not occur for K3 surfaces.

 We use these results to deduce the following statement on generic base point freeness (which can be thought
 of a partial generalization of Corollary \ref{cor:bpfifno0-class} to  K3$^{[2]}$-type varieties):

\begin{theorem*}[{compare Theorem \ref{thm:genbpf4} and Theorem \ref{thm:genbpf!}}]
    Consider the moduli space $\cM_{d,m}$ of 
    polarized irreducible symplectic varieties of K3$^{[2]}$-type (i.e.,\ the space parametrizing pairs
    $(X,A)$ where $X$ is an irreducible symplectic variety of K3$^{[2]}$-type, and $A$ is a primitive, ample
    line bundle, which satisfies $q(A)=2d$ and $\div(A)=m$).
%
    For a generic pair $(X,A)\in \cM_{d,m}$ the line bundle $A$ is base point free.
\end{theorem*}

We start by presenting some important facts on irreducible symplectic varieties in Section
\ref{sec:defISV+basics}. 

In Section \ref{sec:BPinexample}, we use general techniques to study the
particular examples $X$ and $X'$ as above.
On this specific irreducible symplectic variety $X$, we exploit that its birational geometry is
well-understood:
We use results of Bayer--Macr\`i to observe that there is exactly one other birational model $X'$ of $X$
which is an irreducible symplectic variety. Therefore, $X$ and $X'$ are connected
by a chain of Mukai flops (by a theorem of Wierzba).
This enables us to prove base point freeness of all big and nef line bundles on $X$ and $X'$ except from the
line bundle $H+L\in \Pic(X)$ (compare   Proposition
\ref{prop:okinspeccase} and Proposition \ref{prop:mostlyokonX}; see page \pageref{notation} for the definition
of $H$ and $L$). 

The proof of generic base point freeness for all pairs $(d,m)\neq (3,2)$ (see Theorem \ref{thm:genbpf4}) is contained in Section \ref{sec:genbpf4}. It combines the analysis of the
special example with Apostolov's results on deformation equivalence of pairs of polarized irreducible
symplectic varieties of K3$^{[2]}$-type (see Proposition \ref{prop:candeform}). The case of $(d,m)=(3,2)$
is left out, since at this point the base locus of the line bundle $H+L$ with these invariants is not understood.

In Section \ref{sec:geometryEX}, we study the concrete geometry of the example $X$. This allows us to prove in Section
\ref{sec:2dimBL} that $H+L$  has indeed base points along a $\bP^2$ (Theorem \ref{thm:baselocusH+L}).
By analysing the natural scheme structure of the base locus of $H+L$ (in Section \ref{sec:genbpfforH+L}),
we can deduce that a generic deformation of 
$H+L$ is base point free (see Proposition \ref{prop:H+Lgenbpf}). This completes the proof of the generic base
point freeness for $(d,m)=(3,2)$ (Theorem
\ref{thm:genbpf!}).

\section*{Acknowledgements}
Most of this work was part of my dissertation. I would like to thank my supervisor Daniel Huybrechts for
his great supervision.
Further, I am grateful for staying at the ETH-ITS during the finalization of the article (supported by Dr. Max R\"ossler, the Walter Haefner Foundation and the ETH Z\"urich
Foundation).
I would like to thank Daniele Agostini for mentioning that certain line bundles occur as determinants of
tautological bundles.
Last but not least, I thank Michael and Florian for supporting me and my work.


\section*{General conventions}\label{sec:gennotation}
Throughout this article we work over the field $\bC$ of complex numbers. 
The convention for projective bundles is $\bP(F)\coloneqq \underline{\Proj}(\Sym (F\dual))$.

\bigskip

We say that a property holds for a ``generic'' element $t\in T$, if there exists a Zariski-open and dense
subset $U\subseteq T$ such that the property holds for all $t \in U$.

If a property holds for all $t \in T$ outside of a countable union of Zariski-closed (or closed analytic)
subsets of $T$, we say that the property is satisfied for a ``very general'' element.


\section{Definition and basic properties of irreducible symplectic varieties}\label{sec:defISV+basics}
This section contains the definition of irreducible symplectic varieties and some classical
facts about them.

\begin{definition}
  An {\it irreducible symplectic variety} is a simply connected, smooth, projective complex variety $X$ such
  that  $H^0(X,\Omega_X^2)$ is generated by a nowhere degenerate two-form. 
\end{definition}

Irreducible symplectic varieties are also known as ``projective hyperkähler manifolds'' and 
as ``irreducible holomorphic symplectic varieties''.
For an
overview on irreducible symplectic varieties, we refer to \cite[Part III]{Gross-Huybrechts-Joyce},
\cite{O'Grady:Intro}, and
\cite{Huybrechts:HK:basic-results}.

  The existence of a nowhere degenerate two-form implies that the dimension of an irreducible symplectic
  variety $X$ is even.

The $2n$-dimensional Hilbert schemes $S^{[n]}\coloneqq \Hilb^n(S)$ of $n$ points on  K3 surfaces $S$ are
irreducible symplectic varieties (see \cite{Beauville1983}). 
This article mostly deals with irreducible symplectic varieties of {\it K3$^{[n]}$-type} (i.e.~irreducible
symplectic varieties which are 
are deformation equivalent to such Hilb$^n(S)$).

\begin{remark}
  Since an irreducible symplectic variety is simply connected there is a canonical
  identification $\Pic(X)=\NS(X)$. Therefore, we will not distinguish between $\Pic(X)$ and $\NS(X)$ for an irreducible
  symplectic variety $X$.
\end{remark}

The following basic observation plays an important role, whenever one considers birational irreducible
symplectic varieties:

\begin{lemma}[{\cite[Section 2.2]{Huybrechts:birHK_deformations} + \cite[Lemma 2.6]{Huybrechts:HK:basic-results}}] \label{lem:biratisoonH2}
  Let $X$ and $X'$ be birational irreducible symplectic varieties. Then there exist closed subsets
  $Z\subset X$ and $Z'\subset X'$ which are of codimension at least two, such that
  $X\smallsetminus Z\iso  X'\smallsetminus Z'$.
  
  This induces isomorphisms $H^2(X,\bZ)\iso H^2(X',\bZ)$  and $\Pic(X)\iso \Pic(X')$.
  If $L\in \Pic(X)$ and
  $L'\in \Pic(X')$ are
  corresponding line bundles, there is an induced isomorphism on the level of global sections
  $H^0(X,L)\iso H^0(X',L')$.
\end{lemma}

\begin{definition}\label{def:birKbar}
  Let $X$ be an irreducible symplectic variety of dimension $2n$.  Denote its Kähler cone by
  $\cK_X\subseteq H^{1,1}(X,\bR)$. Define its {\it birational Kähler cone} as
  \begin{equation*}
    \cB\cK_X\coloneqq  \bigcup_f f^* (\cK_{X'})\subseteq H^{1,1}(X,\bR),
  \end{equation*}
  where the union is taken over all birational maps $f\colon X\dashrightarrow X'$ from $X$ to another
  irreducible symplectic variety $X'$. Denote its closure by $\birKbar\hspace{-0.4em}_X\subseteq
  H^{1,1}(X,\bR)$.
\end{definition}
 Note that the pullback along $f\colon X\dashrightarrow X'$ is well-defined, since the
  indeterminacy locus is of codimension at least two (see Lemma \ref{lem:biratisoonH2}).

\subsection{Lattice structure on \texorpdfstring{$H^2(X,\bZ)$}{H2(X,Z)}}\label{sec:BBFform}
The second integral cohomology $H^2 (X,\bZ)$ of a hyperkähler manifold $X$ is endowed with a
quadratic form, called the Beauville--Bogomolov--Fujiki form (or  Beauville--Bogomolov form), which we denote by $q$.
For the definition and the proofs of the basic properties we refer again to \cite[Part III]{Gross-Huybrechts-Joyce}. 
Denote the associated bilinear form by $(\,,\,)_q$.

The Beauville--Bogomolov--Fujiki form is an integral and primitive quadratic form on $H^2(X,\bZ)$.
Note however, that the lattice $(H^2(X,\bZ),q)$ is not necessarily unimodular.

The signature of $q$ is $(3, b_2 - 3)$, where $b_2=\rank H^2(X,\bZ)$. Restricted to $H^{1,1}(X,\bR)$ the
signature of $q$ is $(1,b_2-3)$.

Therefore, one can define the {\it positive cone} $\cC_X \subseteq H^{1,1}(X,\bR)$ as the connected component of $\{\alpha \in H^{1,1}(X,\bR) \mid q(a)>0\}$
containing an ample class.


\subsection{Divisibility of elements in \texorpdfstring{$H^2(X,\bZ)$}{H2(X,Z)}}

Recall the following definitions for lattices:

\begin{definition}
  \begin{enumerate}
  \item An element $\alpha$ in a lattice $(\Lambda, q)$ is called {\it primitive}, if it is not a
    non-trivial multiple of another element, i.e.~$\alpha=k\cdot \alpha'$ for $k\in \bZ$ and $\alpha'\in
    \Lambda$ implies that $k=\pm 1$.
  \item    Let $\alpha \in \Lambda$ be an element in a lattice $(\Lambda, q)$. Its {\it divisibility}
    $\div(\alpha)$ is defined as the multiplicity of the element $(\_,\alpha)_q\in \Lambda\dual$,
    i.e.~$\div(A)=m$ if $q(\_,A)=m\cdot v$ for some primitive element $v \in \Lambda\dual$. Equivalently
    $\div(\alpha)=m$ if $m$ spans the ideal $(m)=\{(\beta,\alpha)_q\mid \beta\in \Lambda\} \subseteq
    \bZ$.
  \item Pick a $\bZ$-module basis $\{e_i\}$ of $\Lambda$. The {\it discriminant} of the lattice
    $(\Lambda, q)$ is defined as the determinant of the intersection matrix 
    $\big((e_i,e_j)_q\big)_{i,j}$.
  \end{enumerate}
\end{definition}

\begin{remark}\label{rem:div-discr}
  Note that the divisibility of a primitive element in an arbitrary lattice is always a divisor of its
  discriminant.
\end{remark}

\begin{definition} \label{def:divHK}
  For any element $\alpha \in H^2(X,\bZ)$ in the cohomology of an irreducible symplectic variety, we
  denote  by  $\div(\alpha)$ its divisibility with respect to the lattice $(H^2(X,\bZ),q)$.
  
  Note that in particular for $\alpha \in \Pic(X)\subseteq H^2(X,\bZ)$ we will still use $\div(\alpha)$
  for its divisibility in $H^2(X,\bZ)$. 
\end{definition}

\begin{remark}
  With this notation the divisibility of an element $\alpha \in \Pic(X)$ with respect to the lattice
  $\Pic(X)$ can be bigger than $\div(\alpha)$:
  For instance for a K3 surface $X$ every primitive element in $H^2(X,\bZ)$ has divisibility one, since
  the K3-lattice is unimodular (see e.g.~\cite[Section I.3.3]{HuybrechtsK3}).
  On the other hand, if the K3 surface satisfies $\Pic(X)\iso \bZ\cdot H$ for some line bundle $H$, then $q(H)$ is
  bigger than one (since $H^2(X,\bZ)$ is an even lattice) and thus the divisibility of $H$
  with respect to $\Pic(X)$ would be $q(H)>1$.

  This convention for $\div(\alpha)$ is convenient, because in this way $\div(\alpha)$ is invariant under
  deformation of $X$.
\end{remark}

\subsection{The transcendental lattice}

One defines the transcendental lattice for an irreducible symplectic variety in the following way:
\begin{definition}\label{def:trlattice}
  Let $X$ be an irreducible symplectic variety. Then the transcendental lattice
  $H^2(X,\bZ)_{\tr}\coloneqq T
  \subseteq H^2(X,\bZ)$ is 
  the primitive sublattice, which supports the minimal primitive integral sub-Hodge structure $T$, such that
  $T^{2,0}=H^{2,0}(X)$.
\end{definition}

\begin{lemma}
  With this notation
  \begin{equation*}
    H^2(X,\bZ)_{\tr} = \NS(X)^\perp \subseteq H^2(X,\bZ),
  \end{equation*}
where the orthogonal complement is taken with respect to the Beauville--Bogomolov--Fujiki form.
\end{lemma}
\begin{proof}
  Since $H^{1,1}(X,\bC)$ is orthogonal to $H^{2,0}(X) \oplus H^{0,2}(X)$ with respect to the
  Beauville--Bogomolov--Fujiki form on $H^2(X,\bC)$ (which follows immediately from the definition of
  $q$) and the signature of $H^2(X,\bZ)$ is $(3, b_2-3)$, the same arguments as in 
  \cite[Lemma III.3.1]{HuybrechtsK3} apply.
\end{proof}

\begin{remark}\label{rem:type-argument}
Note that obviously this implies that  $(\alpha,W)_q=0$ for all $\alpha \in H^2(X,\bZ)_{\rm tr}$ and $W\in H^{1,1}(X,\bZ)$.
 On the other hand, if $W \in H^2(X,\bZ)$ is not contained in $H^{1,1}(X,\bZ)$, then there exists an element $\alpha
    \in H^2(X,\bZ)_{\rm tr}$ with $(\alpha,W)_q\neq 0$ (this also follows from the proof of \cite[Lemma III.3.1]{HuybrechtsK3}).

 Furthermore, for a curve $C\subseteq X$, one knows that an element $\sigma \in H^{2,0}(X,\bC)$ satisfies
 $\deg(\sigma|_C)=\int_C \sigma|_C =0$ for type reasons. Since $\deg(\_|_C)$ can be seen as morphism of Hodge
 structures, this implies that $\deg(\beta|_C)=0$ for all $\beta \in H^2(X,\bZ)_{\tr}$.
\end{remark}

\subsection{Mukai flops}\label{ssec:Mukaiflops}
In order to fix the notation and for the convenience of the reader, we present the construction of general Mukai flops as in \cite[§3]{Mukai}, without
giving the proofs of the statements. 

Let $X$ be an irreducible symplectic variety. Suppose $P\subseteq X$ is a closed subvariety of
codimension $r$ in $X$, which is isomorphic to a $\bP^r$-bundle $\pi\colon P\to B$.

Using the non-degenerate holomorphic symplectic two-form $\sigma \in H^0(X,\Omega^2_X)$, one
can show that 
there is an isomorphism
$N_{P|X}\iso \Omega_{P|B}$.
Let $\phi\colon \Xhat \to X$ be the blow-up of $X$ in $P$, and $\phi \colon P'\to B$
 be the dual projective bundle to $P$. Then the exceptional divisor is isomorphic to the coincidence
 variety in $P\times_B P'$:
\begin{equation}\label{eq:EisoZ}
E\iso \bP(N_{P|X})\iso \bP(\Omega_{P|B}) \iso Z\coloneqq \{(p,p') \mid p \in p'\} \subseteq P\times_B
P'. 
\end{equation}

One can observe that
\begin{lemma}\label{lem:MukaiflopOEi}  
   The normal bundle $N_{E|\Xhat}\iso\dO(E)|_E\in \Pic(E)$ corresponds to the line bundle
$\dO(-1,-1)\coloneqq \pr_1^*\dO(-1)\otimes \pr_2^*\dO(-1)\in \Pic(Z)$ via the isomorphism \eqref{eq:EisoZ}.
\end{lemma}
Consequently, the restriction of $N_{E|\Xhat}$  to the fibres of the second projection
$\pr_2\colon Z\to P'$ is $\dO(-1)$.
Therefore, one can apply 
\cite[Corollary 6.11]{Artin:blow-down} to see that there exists an algebraic space $X'$ which is the
blow-down $\phi'\colon \Xhat \to X'$ of $\Xhat$ along $E\iso Z\to P'$.

This induces a diagram of the following form:
\begin{equation*}
\begin{tikzcd}[row sep=small, column sep=tiny]
  && & & E \ar[hook]{dd} \ar{dddllll}[swap]{\eta} \ar{dddrrrr}{\eta'} & & && \\
  &&&&&&&& \\
  && & & \Xh \ar{dl}[swap]{\phi} \ar{dr}{\phi'} & & && \\
 P \ar[hook]{rrr}[swap]{} && &X \ar[dashed]{rr}{f} & 
    &X'\ar[hookleftarrow]{rrr}[swap]{} & &&P'\, ,
\end{tikzcd}
\end{equation*}
where $\Xhat$ is at the same time the blow-up of $X$ in $P$ and the blow-up of $X'$ in $P'$, and
similarly $E$ is at the same time the exceptional divisor for $\phi$ and $\phi'$.
The two maps $\eta$ and $\eta'$ correspond to the two projections on $Z$ via the isomorphism \eqref{eq:EisoZ}.

\begin{definition}
  If $X'$ is a projective variety, then the map $f\colon X\dashrightarrow X'$ is called a {\it general
    Mukai flop}.
  
  In the special case, where $P\iso\bP^n$, and $n= \half \dim (X)$, the map $f$ is called {\it
    elementary Mukai flop} or simply {\it Mukai flop}.
\end{definition}

\begin{remark}\label{rem:omegaofMukaihat}
  The canonical bundle of  $\Xh$  is
  \begin{equation*}
  \omega_{\Xh}=(r-1)E \in \Pic(\Xh).
  \end{equation*}
  See e.g.~{\cite[Exercise II.8.5.(b)]{Hartshorne}}.
\end{remark}

\begin{theorem}[{\cite[Proposition 2.1]{Wierzba}}, see also \cite{Hu-Yau:birational}]\label{thm:Wierzba}
  Let $X$ and $X'$ be birational  irreducible symplectic (projective) 4-folds whose nef cones
  $\Nef(X)\subseteq H^2(X,\bR)$ and
  $f^*(\Nef(X'))\subseteq  H^2(X,\bR)$ are separated by a single wall. 
Denote the indeterminacy locus of $f\colon X\dashrightarrow X'$ by $P$.
Then
\begin{enumerate}
\item the set $P$ is a disjoint union of finitely many copies $P_i$ of $\bP^2$, and
\item the birational map $f \colon X\dashrightarrow X'$ is the Mukai flop in the union of the $P_i$.
\item \label{it:ray}
  Define the extremal ray $R_f\subseteq \NEbar(X)\subseteq N_1(X)$ associated to $f$ in the
  following way: The ray $R_f$ is the unique ray which is orthogonal to the wall separating $\Nef(X)$ and $f^*(\Nef(X'))$. Then 
  \begin{equation*}
    P=\bigcup_{[C]\in R_f} C.
  \end{equation*}
\end{enumerate}
\end{theorem}
\begin{proof}
  If we show that the associated ray $R_f$ is $(X,\eps D)$-negative extremal ray of $\NEbar(X)$ for some
  effective divisor $D$, and $\eps$ such that $(X,\eps D)$ is a klt pair, this follows directly from the statement
  of \cite[Proposition 2.1]{Wierzba} (note that $P$ has automatically codimension at least two by Lemma
  \ref{lem:biratisoonH2}). 

  First, recall that for all effective divisors $(X,\eps D)$ is klt for  $0<\eps\ll 1$.
  Furthermore, since $R_f$ is orthogonal to a wall of  $\Nef(X)$, it is automatically an
  extremal ray. 

  It therefore suffices to find an effective divisor $D\in \Pic(X)$ such that $D$ is negative on $R_f$.
  Since $R_f$ is orthogonal to the wall separating $\Nef(X)$ and $f^*(\Nef(X'))$ there is a rational element $D\in f^*(\Amp(X'))$
  with $(D,R_f)<0$. Then a multiple of $D$ corresponds to a very ample line bundle, which concludes the proof.
\end{proof}

\subsection{The \texorpdfstring{K3$^{[n]}$-type}{K3n-type}} \label{ssec:K3ntype}
We collect several standard facts about irreducible symplectic varieties of
K3$^{[n]}$-type, which we need later.

Fix the notation $\LambdaKE\iso U^{\oplus 3}\oplus E_8(-1)^{\oplus 2}$ for the K3 lattice, where $U$ is
the standard hyperbolic lattice with matrix 
$\left(\begin{smallmatrix}
0&1 \\ 1&0
\end{smallmatrix} \right)$,  and $E_8(-1)$
is the unimodular root-lattice $E_8$ changed by sign.

\begin{proposition}\label{prop:standard-decomposition-K3n}
  Fix a K3 surface $S$. Then it is known that:
  \begin{enumerate}
  \item For every $n\geq 2$, there is an orthogonal decomposition of lattices (preserving the natural
    Hodge structures on both sides):\label{it:H2K3n}
    \begin{equation*}
      H^2(\Hilb^n(S),\bZ)\iso H^2(S,\bZ)\oplus \bZ\cdot \delta,
    \end{equation*}
    where $\delta$ is an integral $(1,1)$-class with $q(\delta)=-2(n-1)$. The class $2\delta$ is represented by
    the Hilbert--Chow divisor.
  \item There is an orthogonal decomposition: \label{it:PicK3n}
    \begin{equation*}
      \Pic(\Hilb^n(S))\iso \Pic(S) \oplus \bZ \cdot \delta.
    \end{equation*}
  \item In particular for each irreducible symplectic variety $X$ of K3$^{[n]}$-type, there exists an
    isomorphism \label{it:H2K3ntyp}
    $H^2(X,\bZ)\iso \LambdaKE \oplus\, \bZ \delta$.
  \end{enumerate}
\end{proposition}
Part \ref{it:H2K3n} and \ref{it:PicK3n} follow from \cite[Proposition 6 + Remark]{Beauville1983}. Part
\ref{it:H2K3ntyp} follows from \ref{it:H2K3n}, 
since for a K3 surface $S$ there is an isomorphism $H^2(S,\bZ)\iso \LambdaKE$ (see \cite[Proposition 3.5]{HuybrechtsK3}).

\begin{definition}
  Fix the notation $\LambdaKEn\coloneqq \LambdaKE \oplus \,\bZ\cdot \delta$ with $q(\delta)=-2(n-1)$, for
  the lattice associated to the second cohomology of an irreducible symplectic variety of K3$^{[n]}$-type.
\end{definition}

\begin{remark}\label{rem:discK3n}
  Since the K3 lattice is unimodular, the discriminant of the lattice
  $\LambdaKEn$ is $2(n-1)$.
\end{remark}

\begin{remark}\label{rem:computediv}
  Proposition \ref{prop:standard-decomposition-K3n} shows that every element in $\alpha \in
  H^2(\Hilb^n(S),\bZ)$ can be expressed by $\alpha=a\lambda+b\delta$ for a primitive element $\lambda\in
  H^2(S,\bZ)$ and some $a,b \in \bZ$.  Since $\LambdaKE$ is a unimodular lattice, this implies that the
  divisibility of $\alpha$ (compare Definition \ref{def:divHK}) is
  \begin{equation*}
    \div(\alpha)= \gcd(a,2b(n-1)).
  \end{equation*}
\end{remark}

\begin{proposition}[{Riemann--Roch for K3$^{[n]}$-type \cite[p.\,188]{Gross-Huybrechts-Joyce}}]\label{prop:EGL}
    Let $X$ be of K3$^{[n]}$-type, and $L \in \Pic(X)$ be a line bundle. Then
  \begin{equation*}
    \chi(X,L)=\bigg(\twolines{\frac{1}{2}q(L)+n+1}{n}\bigg).
  \end{equation*}
\end{proposition}

The following important existence result for Lagrangian fibrations is known to hold for irreducible symplectic
varieties of K3$^{[n]}$-type:
\begin{theorem}[{\cite[Corollary 1.1]{Matsushita13}}]\label{thm:Matsushita}
  Let $X$ be an irreducible symplectic variety of K3$^{[n]}$-type and $0\neq L \in \Pic(X)$ be a primitive nef line bundle with
  $q(L)=0$. Then $\dim h^0(X,L)=n+1$ and $|L|$ induces a Lagrangian fibration  $\phi_L\colon X \to \bP^n$. In
  particular $L$ is base point free.
\end{theorem}

\subsection{Geometry of \texorpdfstring{Hilb$^2$(K3)}{Hilb2(K3)}} \label{ssec:geomofHilb2}

The construction of the Hilbert scheme is particularly easy to describe in the special case of Hilb$^2(S)$:
For a K3 surface $S$ consider the blowup $\SShat$ of $\SxS$ along the diagonal $\Delta_S$. 
The action $\tau$ on $\SxS$, which interchanges the two factors, induces an action $\widehat{\tau}$ on
the blow-up. The Hilbert scheme $\Hilb^2(S)$ is exactly the quotient of $\SShat$ by the action of $\widehat{\tau}$.
It thus comes with natural maps:
\begin{equation*}
  \begin{tikzcd}
    \SxS   &  \SShat \ar{l}[swap]{\zeta_S}\ar{r}{\eps_S}& \Hilb^2(S).
  \end{tikzcd}
\end{equation*}
In this case the embedding $H^2(S,\bZ)\to H^2(\Hilb^2(S),\bZ)$ from Proposition \ref{prop:standard-decomposition-K3n} is given via 
$\alpha_S \mapsto {\eps_S}_*\zeta_S^*\pr_1^*(\alpha_S)\eqqcolon \alpha$.
In particular 
\begin{equation}
\eps_S^*(\alpha)= \zeta_S^*(pr_1^*(\alpha_S) + \pr_2^*(\alpha_S)).\label{eq:epsalpha}
\end{equation}

Usually one defines
$\alpha$ via this property (i.e.~one defines $\alpha$ as the element to which the
$\widehat{\tau}$-invariant element $\zeta_S^*(pr_1^*(\alpha_S) + \pr_2^*(\alpha_S))$ descends in the
quotient $\Hilb^2(S)$). 

If we denote the exceptional divisor of the blow-up $\SShat$ by $E_S$,  one can check that
\begin{equation}
\eps_S^*(\delta)= E_S.\label{eq:epsdelta}
\end{equation}

\begin{lemma}\label{lem:assLBisnef}
  Let $S$ be a K3 surface and $H_S\in \Pic(S)$ a nef line bundle. 
  Then the associated line bundle $H$ is also nef.
\end{lemma}
\begin{proof}
  Let $C\subseteq \Hilb^2(S)$ be an arbitrary curve. Then
  \begin{align*}
    \deg(C.H)=\deg(C.{\eps_S}_* \zeta_S^* \pr_1^*(H_S))
    &= \deg({\pr_1}_* {\zeta_S}_*\eps_S^*C. H_S) \geq 0
  \end{align*}
shows that $H$ is nef.
\end{proof}

\begin{lemma}\label{lem:assLBisbpf}
  Consider a K3 surface $S$, a line bundle $H_S\in \Pic(S)$ and the associated line
  bundle $H \in \Pic( \Hilb^2(S))$. If $H_S$ is base point free then $H$ is also base point free.
\end{lemma}
\begin{proof}
  Start with an arbitrary point $z\in \Hilb^2(S)$. The support of the subscheme $z\subseteq S$ consists of two points
  $x,y \in S$ (set $x=y=\Supp(z)$ in the case where $z$ is non-reduced). Since $H_S$ is base point free, there
  exists a section 
  $s\in H^0(S, H_S)$ with $s(x)\neq 0$ and $s(y)\neq 0$. 
  Consider the section 
  \begin{equation*}
    \zeta^*(pr_1^*(s) \otimes \pr_2^*(s)) \in H^0\Big(\SShat, \zeta_S^*\big(pr_1^*(H_S) + \pr_2^*(H_S)\big)\Big),
  \end{equation*}
which is $\widehat{\tau}$-invariant. It therefore descends to a section $s^{[2]} \in H^0(\Hilb^2(S),H)$,
which satisfies $\eps_S^*(s^{[2]}) = \zeta^*(pr_1^*(s) \otimes \pr_2^*(s))$.
Then $s^{[2]}$ 
takes the  value $s(x)\cdot s(y)\neq 0$ on $z$ and therefore $z$ is not a base point of $H$, which proves the lemma.
\end{proof}

\subsection{Moduli of polarized irreducible symplectic varieties of \texorpdfstring{K3$^{[2]}$-type}{K3n-type}} \label{ssec:moduli-Apostolov}

In this article, we study the base points of certain line bundles on special K3$^{[2]}$-type
varieties.
The following well-known proposition will be crucial for deducing statements for generic elements in the moduli
space of polarized irreducible symplectic varieties of K3$^{[2]}$-type from this statement.
\begin{proposition}[{Corollary to \cite[Proposition 3.2]{Apostolov14}}]\label{prop:candeform}
   Let $X$ and $\Xtild$ be (four-dimensional) irreducible symplectic varieties of K$3^{[2]}$-type.
   Fix primitive classes $D \in \Pic(X)$ and $\Dtild \in\Pic(\Xtild)$. Suppose that
   \begin{compactenum}
      \item $q(D)=q(\Dtild)> 0$, 
      \item $\div(D)=\div(\Dtild)$, and that 
      \item  $D\in \cC_X$ and $\Dtild\in \cC_X$ lie in the positive cones.
   \end{compactenum}
   Then the pairs $(X,D)$ and $(\Xtild', \Dtild)$ are deformation equivalent as pairs (i.e.\,there exists a
   family $\cX'\to T'$ of irreducible symplectic varieties over a possibly singular connected base $T'$,
   with $\cD\in \Pic(\cX')$,
   such that the above pairs are isomorphic to special fibres of
   this family).
\end{proposition}
\begin{proof}
  Starting from the pair $(X,D)$, it is possible to find a family $\cX\to T$ with $\cX_0\iso X$,
  which satisfies that there exists $\cD\in \Pic(\cX)$ such that $\cD|_{\cX_0}=D$, and such that
  $\rho(\cX_t)=1$ for general elements $t\in T$.
  By the projectivity criterion (\cite[Theorem 3.11]{Huybrechts:HK:basic-results} and \cite[Theorem 2]{Huybrechts:HK:basic-results:erratum}) $\cX_t$ is projective and one can deduce
  that $\cD_t$   is ample.

  In the same way deform the pair $(\Xtild, \Dtild)$ to a pair $(\widetilde{\cX}_t,\widetilde{\cD}_t)$ with
  $\widetilde{\cD}_t$ ample.

  Apply Apostolov's result \cite[Proposition 3.2 + Remark 3.3.(1)]{Apostolov14} that the moduli space of polarized irreducible symplectic varieties of
  K3$^{[2]}$-type  (with fixed
  $q$ and $\div$ for the polarization) is connected. This shows that
  the pairs
  $(\cX_t,\cD_t)$ and $(\widetilde{\cX}_t,\widetilde{\cD}_t)$ are deformation equivalent, and concludes the proof.
\end{proof}

\begin{corollary}
  The moduli space $\cM_{d,m}$ which parametrizes pairs $(X,A)$, where $X$ is an irreducible symplectic
  variety of K3$^{[2]}$-type, and $A\in \Pic(X)$ is primitive and ample with $q(A)=2d$ and $\div(A)=m$,
  consists of a unique irreducible component.
\end{corollary}
\begin{proof}
  Apostolov \cite[Proposition 3.2 + Remark 3.3.(1)]{Apostolov14} shows  that this moduli space has only one
  connected component.
  To see that this consists of a single irreducible componen, we use the period map for marked irreducible
  symplectic varieties $\dP \colon \Def(X) \to Q_\Lambda$ as introduced in \cite[Section
    25.2]{Gross-Huybrechts-Joyce} (in the following, we use the notation introduced there).
  For a given point in $\cM_{d,m}$, one can find a small simply connected open neighbourhood $U$ (in the classical
  topology), 
  such that the universal family can be equipped with a marking.
  Furthermore - after potentially shrinking $U$ -  the period map identifies $U$ with
  an open subset of $\rQ_{h^\perp}\subseteq \rQ_{\Lambda}$, where $h\in \Lambda$ corresponds to the class
  of the polarization, and $h^\perp$ is its orthogonal complement in $\Lambda$ (this follows from
  \cite[Th\'eor\`eme 5]{Beauville1983}).
  Since $\rQ_{h^\perp}$ is a smooth quadric in $\bP(h^\perp\otimes \bC)$, this shows that $\cM_{d,m}$ is
  smooth at every point. In particular,  it can only have one irreducible component (since it is connected).
\end{proof}

\section{Base point freeness via the birational model} 
\label{sec:BPinexample}
In this section we study  $X\coloneqq \Hilb^2(S)$ for a K3 surface $S$ with $\Pic(S)\iso \bZ\cdot H_S$ for an ample line bundle $H_S$ with
$(H_S)^2=2$.
 The goal of this section is to prove base point freeness for almost all nef line bundles on $X$
 and its birational models.
This will be a crucial ingredient for the proof of generic base point freeness for K3$^{[2]}$-type (in
particular for the partial result in
Section \ref{sec:genbpf4}). 

While we will focus on the specific geometry of this particular example in Section \ref{sec:2dimBL}, the
techniques in this section are mostly of a very general and combinatorial nature. We consciously keep the
notation more general than necessary for our example, since in this way it can be adapted more easily to
other settings.

Start with the following observation:

\begin{lemma}\label{lem:Hbpfinexample}
  Let $S$ be an arbitrary K3 surface with $\Pic(S)\iso \bZ\cdot H_S$ for an ample line bundle
  $H_S$ (for now, $(H_S)^2>0$ can be arbitrary). 
  Then the associated line bundle $H\in \Pic(\Hilb^2(S))$  is  base point free. Furthermore, $H$ is big and nef.
\end{lemma}

\begin{proof}
  Under the given condition the line bundle
  $H_S$ is base point free by Corollary \ref{cor:bpfifno0-class}.
  Then the associated line bundle
  $H$ is also base point free by Lemma \ref{lem:assLBisbpf}.

   Lemma \ref{lem:assLBisnef} shows that $H$ is nef. Then use the Fujiki relation (\cite[Proposition
     23.14]{Gross-Huybrechts-Joyce}) and 
  $q(H)=(H_S)^2> 0$ to see that 
  $\int_{\Hilb^2(S)}H^{2n}= c^{-1}\cdot q(H)^n>0$, 
  which implies that $H$ is big and nef.
 \end{proof}

From now on, we will consider the case, where $(H_S)^2=2$.
\begin{lemma} \label{lem:cones}
  Let $S$ be a K3 surface with $\Pic(S)\iso \bZ\cdot H_S$ for an ample line bundle $H_S$ with
$(H_S)^2=2$. Consider the irreducible symplectic variety $X\coloneqq \Hilb^2(S)$ with the usual
decomposition $\Pic(X)\iso \bZ\cdot H \oplus \bZ \cdot \delta$, where $H$ is the line bundle in $\Pic(X)$ associated
to $H_S$ (compare Proposition \ref{prop:standard-decomposition-K3n}). Then
\begin{enumerate}
\item $\overline{\cC_X}\cap \Pic(X)_\bR= \langle H+\delta, H-\delta\rangle$ \label{it:cones-a}
\item  $\birKbar_X \cap \Pic(X)_\bR= \langle H, H-\delta \rangle$, \label{it:cones-b}
\item $\Nef(X)= \langle H, 3 H - 2 \delta \rangle \subseteq \Pic(X)_\bR$, and \label{it:cones-c}
\item \label{it:cones-d} there is a unique other birational model $X'$ of $X$ which is an irreducible
  symplectic variety. This satisfies
  \begin{equation*}
\Nef(X')= \langle
  3H'-2\delta', H'-\delta' \rangle \subseteq \Pic(X')_\bR,
\end{equation*}
 where $H',\delta'\in \Pic(X')$ are the line bundles which correspond to $H$ and $\delta$ via
  the birational transform.
\end{enumerate}
\end{lemma}
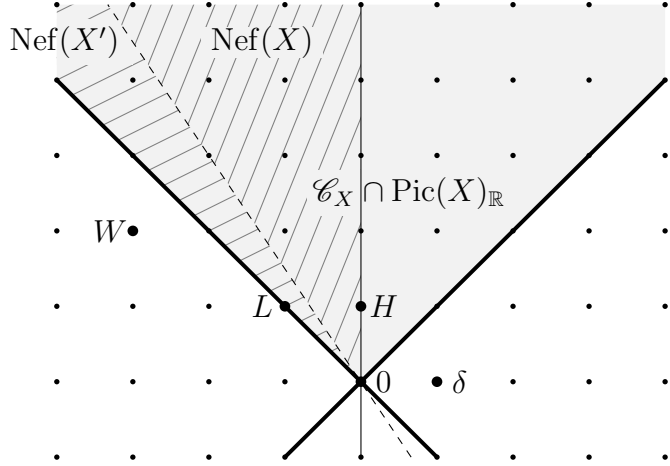
\begin{figure}[htbp] 
  \centering

\definecolor{light-gray}{gray}{0.95}
  \begin{tikzpicture}[scale=1]

       \path[fill=light-gray] (0,0) -- (-4,4) -- (-4,5) -- (4,5) -- (4,4) -- cycle;

       \begin{scope}
         \clip (0,0) -- (-3.33,5) -- (0,5) -- cycle;
      \foreach \x in {0,1,...,20} 
      {\draw [gray] (0-0.2*\x,0+0.2*\x) -- (2-0.2*\x,5+0.2*\x);}
       \path[fill=light-gray] (-2.1,4.3) rectangle (-0.5,4.7);
       \path[fill=light-gray] (-0.7,2.2) rectangle (0,2.8);
       \end{scope}

       \begin{scope}
         \clip (0,0) -- (-3.33,5) -- (-4,5) -- (-4,4) -- cycle;
      \foreach \x in {0,1,...,30} 
      {\draw [gray] (0-0.2*\x,0+0.2*\x) -- (4-0.2*\x,2+0.2*\x);}
       \path[fill=light-gray] (-4,4.2) rectangle (-3,4.8);
       \end{scope}

     \foreach \x in {-4,-3,...,4} 
       \foreach \y in {-1,...,5}{
         \fill(\x,\y) circle (1pt);}
       \fill(0,0) circle (2 pt);
       \node (o) at (0.3,0) {$0$};
       \fill(1,0) circle (2 pt);
       \node (d) at (1.3,0) {$\delta$};
       \fill(0,1) circle (2 pt);
       \node (h) at (0.3,1) {$H$};
       \fill(-1,1) circle (2 pt);
       \node (l) at (-1.3,1) {$L$};
       \fill(-3,2) circle (2 pt);
       \node (w) at (-3.3,2) {$W$};

\draw (0,-1) -- (0,5);
\draw [line width=1.5pt] (-1,-1) -- (4,4);
\draw [line width=1.5pt] (1,-1) -- (-4,4);

\draw [dashed] (0.66,-1) -- (-3.33,5);

       \node (Cx) at (0.63,2.5) {$\cC_X\cap \Pic(X)_\bR$};
       \node (nefx) at (-1.3,4.5) {$\Nef(X)$};
       \node (nefx') at (-3.9,4.5) {$\Nef(X')$};

  \end{tikzpicture}

  \caption{Depiction of cones in $\Pic(X)_\bR$}
  \label{fig:cones}
\end{figure}

\begin{proof}
  For \ref{it:cones-a}, it is sufficient to note that $H$ is nef (by Lemma \ref{lem:Hbpfinexample}), $q(H+\delta)=0$ and $q(H-\delta)=0$.

  Part \ref{it:cones-b} was already proved by Bayer and Macr\`i in \cite[Proposition
  13.1.(a)]{BayerMacri13}, and part 
  \ref{it:cones-c} in \cite[Lemma 13.3.(b)]{BayerMacri13}.
  Part \ref{it:cones-d} can be observed using the same arguments. 

For an alternative proof one can use Mongardi's description of the wall divisors for K3$^{[2]}$-type (see
\cite[Proposition 2.12]{Mongardi13}). This implies that the only wall divisors on $X$ are $\pm
\delta$ and $\pm (3H \pm 2\delta)$, from which one can deduce that the chamber decomposition is as claimed.
\end{proof}

For the rest of this section fix the following 
\begin{notation*}\label{notation}
  Let $X\coloneqq \Hilb^2(S)$ for a K3 surface $S$ with $\Pic(S)\iso \bZ\cdot
  H_S$ for an ample line bundle $H_S$ with $(H_S)^2=2$. 
  By $H\in \Pic(X)$ we refer to the line bundle associated to $H_S$.
  We consider the standard decomposition $\Pic(X)\iso \bZ\cdot H \oplus \bZ\cdot \delta$ with $q(\delta)=-2$.
  Define $L\coloneqq H-\delta$, and $W\coloneqq 2H-3\delta$.

  The unique other birational model of $X$ which is an irreducible symplectic variety is called $X'$. The line bundles $H', \delta', L'$, and $W'\in
  \Pic(X')$ are the ones corresponding to $H, \delta, L$, and $W$ via the birational transform.
\end{notation*}

\begin{remark}\label{rem:WisWall}
The line bundle $W$ is one of the wall divisors mentioned in the last proof. It has the key
property that its orthogonal complement in  $\Pic(X)_\bR$ is exactly the wall between 
$\Nef(X)$ and $\Nef(X')$, since $(3H-2\delta,W)_q
=0$.
\end{remark}

Note that the integral part $\birKbar_X\cap \Pic(X)$ is generated by $L$ and $H$. The line bundle $H\in \Pic(X)$ is base point free by
Lemma \ref{lem:Hbpfinexample}. Furthermore, observe that $q(L')=q(L)=0$. By Matsushita's result (see Theorem
\ref{thm:Matsushita}), this implies that 
$L'\in \Pic(X')$ is base point free.

By studying the link between $X$ and $X'$, we will be able to prove base point freeness for many 
line bundles 
in $\Nef(X)\cap\Pic(X)$ and $\Nef(X')\cap \Pic(X')$, using the base point freeness of $H$ and $L'$. 
Theorem \ref{thm:Wierzba} applies to $X$ and $X'$ as above and shows that they are connected by a
sequence of Mukai flops.
According to Theorem \ref{thm:Wierzba}, denote the components of the indeterminacy locus of
$f\colon X \dashrightarrow X'$ by $P_i\iso \bP^2$ with $i=1,\dots ,r$.

Since $f$ is induced by the Mukai flop in the union of the $P_i$ (compare Section \ref{ssec:Mukaiflops}), it fits into the following diagram:
\begin{equation*}
\begin{tikzcd}[row sep=small, column sep=tiny]
  && & &\bigsqcup_i E_i \ar[hook]{dd} \ar{dddllll}[swap]{\eta} \ar{dddrrrr}{\eta'} & & && \\
  &&&&&&&& \\
  && & & \Xh \ar{dl}[swap]{\phi} \ar{dr}{\phi'} & & && \\
 P = \bigsqcup_i P_i \ar[hook]{rrr}[swap]{} && &X \ar[dashed]{rr}{f} & 
    &X'\ar[hookleftarrow]{rrr}[swap]{} & &&\bigsqcup_i P'_i =P'\, ,
\end{tikzcd}
\end{equation*}

where $\Xhat$ is the blow-up of $X$ in the union of the $P_i$ and the $E_i$ are the corresponding
exceptional divisors. Note that by the construction of a Mukai flop, $\Xhat$ is at the same time the
blow-up of $X'$ in the union of $P_i'\iso \bP^2\dual$, 
where $P'\coloneqq \sqcup_i P_i'$ is the indeterminacy locus of $f^{-1}$, and again the $E_i$ are the
corresponding exceptional divisors.
With this notation $E_i$ is isomorphic to the coincidence variety $E_i\iso Z\coloneqq \{(p,p')\mid p\in
p'\} \subseteq P_i\times P_i' \iso \bP^2 \times \bP^2\dual$. 
We use the following standard notation $\dO(a,b)\coloneqq \dO(a)\boxtimes \dO(b)\coloneqq
\pr_1^*\dO(a)\otimes \pr_2^*\dO(b)\in \Pic(\bP^2 \times \bP^2\dual)$ for $a,b \in \bZ$.

In the following we prove some combinatorial properties for pullbacks of line bundle on $X$ or $X'$ to $\Xh$.

\begin{lemma}\label{lem:degphiA}
   With the notation fixed on page \pageref{notation}, there exist constants $m_i \in \bQ_+$ for $i=1,\dots, r$ such that
  for all $i$ and for an arbitrary line bundle $A\in \Pic(X)$ 
     \begin{equation*}
       \Pic(E_i)\to \Pic(Z), \quad \phi^*(A)|_{E_i} \mapsto \pr_1^*\dO( m_i \cdot (A,W)_q).
     \end{equation*}
These $m_i$ do not depend on $A$.
For each $i$ the constant $m_i$ is bounded by $m_i\geq \frac{1}{2}$.
\end{lemma}

\begin{proof}
Begin by observing that $\Pic(\bP^2 \times \bP^2\dual)\iso \Pic(\bP^2) \times \Pic(\bP^2\dual) \iso \bZ
\times \bZ$ (see e.g.~\cite[Exercise III.12.6]{Hartshorne}). 
Therefore, also $\Pic(Z) \iso \bZ\times \bZ$ by Lefschetz hyperplane theorem, which applies since $Z \in
|\dO(1,1)|$ is a smooth ample divisor. Therefore, $\phi^*(A)|_{E_i}$ corresponds to  a line bundle of the form
$\dO(s_i,t_i)$ on $Z$ for some integers $s_i$ and $t_i$.

Consider a line $C_i\subseteq P_i$ and  the induced line 
$\Chat_i\coloneqq \{(p_i,C_i)\mid p_i\in C_i\}\subseteq E_i$. 
Since $\phi^*(A)$ is pulled back from the first factor, $\phi^*(A)|_{E_i}$ corresponds to $\dO(s_i,0)$ for some $s_i\in \bZ$.

On the other hand
\begin{equation*}
  \deg \phi^*(A)|_{\Chat_i} =\deg A|_{\phi_*\Chat_i} = \deg A|_{C_i},
\end{equation*}
and therefore $\phi^*(A)|_{E_i}$ corresponds to $\dO(\deg A|_{C_i},0)=\pr_1^*\dO(\deg A|_{C_i})\in \Pic(Z)$. 

By Theorem \ref{thm:Wierzba}.\ref{it:ray}, the line $C_i$ represents a curve class associated to $f\colon X\to X'$ and thus
is orthogonal to the wall separating $\Amp(X)$ and $\Amp(X')$, i.e.\, it satisfies $\deg(3H-2\delta|_{C_i})=0$. 
On the other hand, we observed in Remark \ref{rem:WisWall} that  $(3H-2\delta,W)_q =0$. 
Consequently, there exists $m_i\in \bQ$ such that 
\begin{equation*}
\deg(\_ |_{C_i})= m_i\cdot (\_,W)_q\in \Pic(X)\dual.
\end{equation*}

 Note that automatically $m_i>0$, since $(H,W)_q=4>0$
 and $\deg(H|_{C_i})>0$ because $H$ is big and nef and does not vanish on $C_i$. 

In particular $\phi^*(A)|_{E_i}$ corresponds to $\pr_1^*\dO(\deg A|_{C_i})=\pr_1^*\dO(m_i \cdot (A,W)_q)$, as claimed.

In order to bound $m_i$, first  observe that $\div(W)=2$ (compare Remark \ref{rem:computediv}). Further
note that  $\deg(\_ |_{C_i})$ and $(\_,W)_q$ both vanish on the 
transcendental lattice of $X$ (see Remark \ref{rem:type-argument}). Therefore, $\deg(\_ |_{C_i})= m_i\cdot (\_,W)_q\in H^2(X,\bZ)\dual$ (and
not only in $\Pic(X)\dual$). 
 Therefore, $m_i \in \frac{1}{2}\cdot \bZ$ and thus $m_i\geq \frac{1}{2}$. 
\end{proof}

\begin{corollary} 
  \label{cor:compareLB}
  Consider an arbitrary line bundle $A\in \Pic(X)$ and  the associated line bundle $A'\in \Pic(X')$. Let
  $m_i$ for $i=1,\dots,r$ be the same constants as in Lemma \ref{lem:degphiA}. Then the following
  relation holds: 
  \begin{equation*}
    \phi'^*(A')=\phi^*(A) + \sum_{i=1}^r m_i \cdot (A,W)_q\cdot E_i.
  \end{equation*}
\end{corollary}
\begin{proof}
The variety $\Xhat$ is the blow-up of $X$ in the $P_i$ (for $i=1,\dots,r$). Therefore
\begin{equation*}
\Pic(\Xhat)\iso \Pic(X)\oplus \bigoplus_{i=1}^r E_i.
\end{equation*}

Since $\phi^*(A)$ and $\phi'^*(A')$ coincide outside the exceptional loci, they only differ by multiples
of the $E_i$, i.e.\, there are integers $\gamma_i$ such that 
\begin{equation*}
\phi'^*(A')= \phi^*(A) + \sum_{i=1}^r \gamma_i\cdot  E_i.
\end{equation*}
The pullback $\phi'^*(A')|_{E_i}$ corresponds to $\dO(0,t_i')$ for some $t_i'\in \bZ$, since it is trivial
on the fibres of $\phi'$. On the other hand $\phi^*(A)|_{E_i}$ corresponds to $\dO( m_i \cdot (A,W)_q ,0)$ by
Lemma \ref{lem:degphiA}. Furthermore it is known that $E_i|_{E_i}$ corresponds to $\dO(-1,-1)$ via the
isomorphism $E_i\iso Z\subseteq \bP^2 \times \bP^2\dual$ (compare Lemma \ref{lem:MukaiflopOEi}), and of course $E_j|_{E_i}$ is trivial for all $i\neq j$.  

Comparing the degrees on $E_i$ thus gives that $\gamma_i=m_i \cdot (A,W)_q$. This
completes the proof.
\end{proof}

\begin{corollary} \label{cor:degphiA'}
Let $A'\in \Pic(X')$ be an arbitrary line bundle. Then the same  $m_i$ as in Lemma \ref{lem:degphiA}
($i=1,\dots,r$) satisfy
  \begin{equation*}
    \Pic(E_i)\to \Pic(Z), \quad \phi'^*(A')|_{E_i} \mapsto \pr_2^*\dO(- m_i \cdot (A',W')_q).
  \end{equation*}
\end{corollary}
\begin{proof}
 Let $A \in \Pic(X)$ be the line bundle corresponding to $A'$ and fix $i\in \{1,\dots,r\}$. Then Corollary
 \ref{cor:compareLB} shows that  
\begin{align*}
  \phi'^*(A')|_{E_i}=\phi^*(A)|_{E_i}+ \sum_{j=1}^r m_j \cdot (A,W)_q\cdot E_j|_{E_i}.
\end{align*}
By Lemma \ref{lem:degphiA} the right hand side corresponds to
\begin{align*}
  \dO(& m_i\cdot (A,W)_q ,0) + m_i \cdot (A,W)_q\cdot \dO(-1,-1)\\
  &= \dO( 0,-m_i \cdot (A,W)_q)
   =\pr_2^*\dO(-m_i \cdot (A',W')_q)
\end{align*}
as claimed.
\end{proof}

\begin{proposition}\label{prop:okinspeccase}
Let $S$ be a K3 surface with $\Pic(S)\iso \bZ\cdot H_S$ for an ample line bundle $H_S$ with
$(H_S)^2=2$. Let $X$ be $\Hilb^2(S)$, and $X'$ be  the unique other irreducible symplectic variety which is
birational to $X$ (as in Lemma \ref{lem:cones}). Then all nef line bundles on $X'$ are base point free.
\end{proposition}

\begin{proof}
  Fix a nef line bundle $A'\in \Pic(X')$ and denote the associated line bundle on $X$ by $A\in
  \Pic(X)$. Since $L'$ and $H'$ span all integral elements in the birational Kähler cone
  $\birKbar_{X'}\cap \Pic(X')_\bR=\langle H', L'\rangle$, there exist non-negative
  integers $a,b \in \bZ_{\geq 0}$ for which $A'=aH'+bL'$.  Since $A'$ is nef by assumption, $A'\in
  \Nef(X')=\langle H'+2L',L'\rangle$ (compare Lemma \ref{lem:cones}.\ref{it:cones-d}), and therefore
  \begin{equation}\label{eq:ab}
a \leq \frac{1}{2}b.
\end{equation}

  In the case $a=0$, observe that $A'$ is a multiple of $L'$ and therefore base point free by Theorem
  \ref{thm:Matsushita}. Thus we may assume that $a>0$.

  Since $\phi'$ is a dominant morphism of projective varieties with connected fibres the pull-back
  $\phi'^*\colon H^0(X',A')
  \overset{\iso}{\too} H^0(\Xhat, \phi'^*(A'))$ is an isomorphism.
  Therefore, it suffices to prove that $\phi'^*(A')\in \Pic(\Xhat)$ is base point free, in order to show base point
  freeness of $A'\in \Pic(X')$.
  Working on $\Xhat$ instead, has the advantage that we know $\phi'^*(L')$ and $\phi^*(H)$ are both base
  point free (compare Theorem \ref{thm:Matsushita} and Lemma \ref{lem:Hbpfinexample}).

  Use Corollary \ref{cor:compareLB} to see
  \begin{align}
    \phi'^*(A')=  b\cdot\phi'^*(L') + a\cdot \phi'^*(H')
    = b\cdot \phi'^*(L') + a \cdot \phi^*(H)+ a\cdot \sum_{i=1}^r 4 m_i \cdot E_i \, . \label{eq:pullbackA}
  \end{align}
The fact that $\phi'^*(L')$ and $\phi^*(H)$ are base point free immediately implies that the base
locus of $\phi'^*(A')$ is contained in the union of the $E_i$.

 Note that 
\begin{equation}\label{eq:(A,W)q}
(A', W')_q=(aH'+ bL',2H'-3\delta')_q=4a-2b \leq 0,
\end{equation}
where the last inequality follows from \eqref{eq:ab}.
Consequently by Corollary \ref{cor:degphiA'} the restriction of $\phi'^*(A')|_{E_i}$ corresponds to
\begin{equation*}
   \dO(0, -m_i\cdot (A',W')_q)=\dO(0,\underbrace{-m_i\cdot (4a-2b)}_{\geq 0}),
\end{equation*}
and thus the restriction $\phi'^*(A')|_{E_i}$ is base point free.

On the other hand, consider the restriction sequence
\begin{equation*}
  0\to \dO\big( \phi'^*(A')-\sum_{i=1}^r E_i\big) \to \phi'^*(A') \to \bigoplus_{i=1}^r \phi'^*(A')|_{E_i}\to 0 \quad
  \in \Coh(\Xh). 
\end{equation*}
This induces an exact sequence
\begin{equation*}
  H^0(\Xhat, \phi'^*(A')) \to \bigoplus_{i=1}^r H^0\big(E_i,\phi'^*(A')|_{E_i}\big) \to H^1\Big(\Xhat, \phi'^*(A')-\sum_{i=1}^r E_i\Big).
\end{equation*}
Therefore, it is sufficient to prove that $H^1(\Xhat, \phi'^*(A')-\sum_i E_i)=0$, since then the above
restriction is surjective. This implies that
$\phi'^*(A')$ has 
no base points in the union of the $E_i$, which was the only thing left to see.

Note that Remark \ref{rem:omegaofMukaihat} implies that $\omega_{\Xhat}=\sum_i E_i$.
By Kodaira vanishing the cohomology group $H^1(\Xhat, \phi'^*(A')-\sum_i E_i)$ vanishes if
$\phi'^*(A')-\sum_i E_i -
\omega_{\Xhat} = \phi'^*(A')-2 \sum_i E_i$ is big and nef. 

For the rest of the proof we show that $\phi'^*(A')-2 \sum_i E_i$ is big and nef.

First note that by \eqref{eq:pullbackA}, we know that 
\begin{equation*}
\phi'^*(A')-2\sum_{i=1}^r E_i = b\cdot \phi'^*(L') + a \cdot \phi^*(H)+  \sum_{i=1}^r (4 m_i a - 2 )\cdot E_i.
\end{equation*}
To see that this is big, observe that  $a\cdot\phi^*(H)$ is big since $a> 0$ and $H$ is big (by Lemma \ref{lem:Hbpfinexample}), and note that $4m_ia-2\geq 0$, since $m_i\geq \frac{1}{2}$.

Thus we only need to check that $\phi'^*(A')-2 \sum_i E_i$ is nef. 
Let $C\subseteq \Xhat$ be a curve. 
If $C\subseteq E_{i_0}$ lies in the exceptional divisor for some $i_0$, then 
\begin{equation*}
  \deg((\phi'^*(A')-2 \sum_i E_i)|_C)> 0,
\end{equation*}
 since by Corollary \ref{cor:degphiA'}   the
restriction $\big(\phi'^*(A')-2 \sum_i E_i\big)|_{E_{i_0}}$ corresponds to 
\begin{equation*}
 \dO(2, -m_{i_0}\cdot (A',W')_q+2) 
\overset{\eqref{eq:(A,W)q}}{=}\dO(2,\underbrace{-m_{i_0}\cdot (4a-2b)}_{\geq 0}+2).
\end{equation*}

Suppose on the other hand that $C\subseteq \Xhat$ is not contained in one of the exceptional divisors.
Then the image of $C$ under $\phi$ is a non-trivial curve and thus use Corollary \ref{cor:compareLB} and
projection formula to see:
\begin{align*}
0< \deg(H.\phi_*C)
&=\deg(\phi^* H. C) 
= \deg\Big(\big(\phi'^*(H')- \sum_{i=1}^r 4 m_i \cdot E_i\big). C\Big)\\
&= \deg(H'.\phi'_*C) - \sum_{i=1}^r 4 m_i \deg(E_i.C).
\end{align*}

In particular this implies that 
\begin{align*}
\deg\Big( \big(&\phi'^*(A')-2 \sum_{i=1}^r E_i\big).C\Big)
= a \underbrace{\deg(\phi'^*H'.C)}_{=\deg(H'.\phi'_*C)} + b \deg(\phi'^*L'.C) - 2 \sum_{i=1}^r \deg(E_i.C)\\
&> \sum_{i=1}^r (4 a m_i -2) \deg(E_i.C)+ b \deg(\phi'^*L'.C).
\end{align*}

 Note that $4am_i-2\geq 0$, since $a>0$ and $m_i\geq \frac{1}{2}$.
Since furthermore $\deg(E_i.C)\geq 0$ for all $i$, because $C$ is not
 contained in any of the $E_i$, the left term is at least zero.
 Additionally, the right term is at least zero, since $L'$ is a nef line bundle on $X'$, and thus $\deg(\phi'^*L'.C) = \deg(L'.\phi'_*C) \geq 0$.
 This shows that $\phi'^*(A')-2 \sum_i E_i$ is nef, and concludes the proof.
\end{proof}

We also prove the following proposition:

\begin{proposition}\label{prop:mostlyokonX}
  Let $X$ be as in Proposition \ref{prop:okinspeccase}. Then all nef line bundles $A =aH+bL\in \Pic(X)$, which
  satisfy $aH+bL \neq H+L$ are base point free.
\end{proposition}

\begin{proof}
  First consider the case $b\neq 1$. Then the proof relies on similar arguments as the proof of Proposition
  \ref{prop:okinspeccase}. 

Note that an arbitrary nef line bundle can be expressed in the form $A=aH+bL$. Applying Lemma
\ref{lem:cones}, the condition that
$A$ is nef yields
\begin{equation}\label{eq:2ageqb}
  a \geq \frac{1}{2} b.
\end{equation}
We may assume that $b\neq 0$, since for multiples of $H$, we already know base point freeness by
Lemma \ref{lem:Hbpfinexample}. Since we further assumed $b\neq 1$, we are in the case $b\geq 2$.

Again, we may check base point freeness of $\phi^*(A)$ instead.

Using Corollary \ref{cor:compareLB} gives
  \begin{align}
    \phi^*(A)=a\cdot \phi^*(H) + b\cdot\phi^*(L)
    =  a \cdot \phi^*(H)+ b\cdot \phi'^*(L') +b\cdot \sum_{i=1}^r 2 m_i \cdot E_i \, . \label{eq:phiA}
  \end{align}
In particular the base locus of $\phi^*(A)$ is contained in the $E_i$.

By Lemma \ref{lem:degphiA} the restriction $\phi^*(A)|_{E_i}$ corresponds to
\begin{equation}\label{eq:restrA}
  \dO\big(m_i\cdot (A,W)_q,0\big)= \dO(\underbrace{4a-2b}_{\geq 0},0),
\end{equation}
which is globally generated since by \eqref{eq:2ageqb} both integers are at least zero.

Thus we only need to show that the restriction
\begin{equation*}
H^0(\Xhat, \phi^*(A)) \to \bigoplus_{i=1}^r H^0\big(E_i,\phi^*(A)|_{E_i}\big)
\end{equation*}

is surjective. As in the proof of
Proposition \ref{prop:okinspeccase} one can use Kodaira vanishing to see that it suffices to prove 
that $\phi^*(A)-2\sum_iE_i$ is big and nef. 

Since by assumption $b>1$, 
bigness of $\phi^*(A)-2\sum_iE_i$ follows from \eqref{eq:phiA}, since under these conditions
$2m_ib-2 \geq 0$. 

For nefness, consider a curve $C\subseteq \Xhat$, and distinguish again between two cases: 
If $C\subseteq E_{i_0}$ for some $i_0$, then deduce from \eqref{eq:restrA}, that $\phi^*(A)-2\sum_iE_i$
restricts with positive degree to $C$.

Otherwise, if $C$ is not contained in the exceptional locus of $\phi$, use nefness of $L'\in \Pic(X')$ to see
\begin{align*}
  0 \leq \deg(L'.\phi'_*C)
  &=\deg(\phi'^*L'.C) 
  = \deg(\phi^*L.C) + \sum_{i=1}^r  m_i \cdot (L,W)_q\cdot\deg(E_i.C)\\
  &=\deg (\phi^*L.C) - \sum_{i=1}^r 2m_i \cdot \deg(E_i.C).
\end{align*}

This implies that 
\begin{align*}
  \deg\Big((\phi^*A-2\sum_{i=1}^r E_i).C\Big)
  &= a \deg(\phi^*H.C) + b \deg(\phi^*L.C) - 2 \sum_{i=1}^r \deg (E_i.C)\\
  &\geq a \deg(H.\phi_*C) + \sum_{i=1}^r (2m_i\cdot b-2)\deg(E_i.C).
\end{align*}
Since $b>1$ both terms are at least zero, and thus $\phi^*A-2\sum_iE_i$ is nef. This concludes the proof in
the case $b\neq 1$.

For the case where $b=1$ (and $a\geq 2$, since $a=1$ is excluded in the statement of the theorem), we use the
tautological bundle $(kH_S)^{[2]}$, which is defined in the following way:
Consider $X=Hilb^2(S) \overset{p}{\leftarrow} Hilb^2(S)\times S\overset{q}{\to}S$, and let
$\Xi \subseteq Hilb^2(S)\times S$ be the universal family.
Then $(kH_S)^{[2]}\coloneqq p_*(\dO_{\Xi}\otimes q^*(H_S^{\otimes k}))$.
It is known that 
$\det ((kH_S)^{[2]})=kH-\delta = (k-1)H+L$ (see \cite{Lehn99}: Lemma 3.7 and the $c_1$-part of Theorem 4.6).
Furthermore, it is a basic observation, that $(kH_S)^{[2]}$ is generated by its global sections if $kH_S$ is
very ample. In this case $\det((kH_S)^{[2]})=(k-1)H+L$ is base point free.
Since $kH_S$ is very ample for all $k\geq 3$ (\cite[Theorem 8.3]{Saint-Donat74}), this implies that $aH+L$ is base point free for $a\geq 2$.
\end{proof}

\begin{remark}\label{rem:halfisproblem}
  Note that instead of assuming $b\neq 1$, we could theoretically also restrict ourselves to the case
  $m_i\neq \half$ in order to ensure the necessary positivity. However, $m_i= \half$ actually occurs in
  our example
  (compare Lemma \ref{lem:miishalf}).
\end{remark}

\section{First theorem on generic base point freeness}\label{sec:genbpf4}
In this section we prove  base point freeness for almost all polarizations on K3$^{[2]}$-type generically in
the moduli (see
Theorem \ref{thm:genbpf4}). This combines deformation results with the study of the particular example in
the last section.

\begin{definition}
  Let $\cM_{d,m}$ be the moduli space of $(d,m)$-polarized irreducible symplectic varieties of
  K3$^{[2]}$-type, i.e.~the space parametrizing pairs $(X,A)$, where $X$ is an irreducible symplectic
  variety of K3$^{[2]}$-type, and $A\in \Pic(X)$ is primitive and ample with $q(A)=2d$ and $\div(A)=m$.
\end{definition}
  The goal of this section is to prove the following theorem:
\begin{theorem}\label{thm:genbpf4}
Suppose $(d,m)\neq (3,2)$, then for generic pair $(X,A)\in \cM_{d,m}$ the line bundle $A\in \Pic(X)$ is
base point free (if $\cM_{d,m}\neq \emptyset$).
\end{theorem}

\begin{remark}
  As an immediate consequence, for a generic $(X,A)\in \cM_{d,m}$ with $(d,m)\neq (3,2)$ the line bundle
  $kA$ is base point free for arbitrary $k>0$.
\end{remark}

\begin{remark}
  Theorem \ref{thm:genbpf4} can be formulated in a different way:

  Let $X$ be an irreducible symplectic variety of K3$^{[2]}$-type and $A\in \cC_X\cap\Pic (X)$. Then a
  generic deformation of the pair $(X,A)$ will satisfy that the line bundle is base point free, unless 
  $q(A)=6k^2$ and $\div(A)=2k$ for some $k\in \bN$. Note that the last condition is equivalent to $A=k\cdot A'$ for
  primitive $A'$ with
  $q(A')=6$ and $\div(A')=2$.

  This reformulation follows from Proposition \ref{prop:candeform}.
\end{remark}

For the proof of this theorem we need two lemmas which provide ``nice'' pairs with given numerical properties.

The discriminant of the lattice $\Lambda\coloneqq\LambdaKEn$  is
two by Remark \ref{rem:discK3n}. Therefore, there are two possible values for the divisibility of a primitive line bundle $A$: Either
$\div(A)=1$ or $\div(A)=2$ (compare Remark \ref{rem:div-discr}). These two cases will be dealt with separately.

\subsection{Case 1: \texorpdfstring{$\div(A)=1$}{div(A)=1}}
\begin{lemma}\label{lem:exnice1}
  Fix $d\in \bZ>0$ and $m=1$.
  Then there exists a K3 surface $S$ with
  $\Pic(S)\iso \bZ\cdot A_S$ and $A_S\in \Pic(S)$ is ample, such that the associated line bundle
  $A\in\Pic(\Hilb^2(S))$ satisfies
  \begin{compactenum}
  \item $q(A)=2d$ and 
  \item $\div(A)=m=1$.
  \end{compactenum}
\end{lemma}

\begin{proof}
  Pick a K$3$ surface $S$ such that $\Pic(S)\iso \bZ\cdot A_S$ for a line bundle $A_S$ with
  $(A_S)^2=2d>0$. Then either $A_S$ or $-A_S$ is ample, so we may assume that $A_S$ is ample. Let $A \in
  \Pic(\Hilb^2(S))$ be the associated line bundle to $A_S$. 
 Then by construction $A$ satisfies $q(A)= (A_S)^2 = 2d$, and $\div(A)=1=m$ by Remark \ref{rem:computediv}.
\end{proof}

\subsection{Case two: \texorpdfstring{$\div(A)=2$}{div(A)=2}}
\begin{lemma}\label{lem:exnice2}
  Pick $d\in \bZ$ such that for $m=2$ the moduli space $\cM_{d,m}$ is non-empty.
  Consider a K3 surface $S$ with $\Pic(S)=\bZ\cdot H_S$ for an ample line bundle with $(H_S)^2=2$. 
  Set $X\coloneqq \Hilb^2(S)$. Then there exists $A\in \Pic(X)$ such that 
  \begin{compactenum}
    \item $q(A)=2d$ and 
    \item $\div(A)=m=2$.
  \end{compactenum}
  In fact, with the notation of page \pageref{notation}, $A$ can be chosen as 
  $H+(2k-1)L$ for some $k\in \bZ_{>0}$.
\end{lemma}

\begin{proof}
  By assumption there exists a pair $(\Xtild, \Atild)\in \cM_{d,m}$. Consider the usual 
   lattice isomorphisms $H^2(\Xtild,\bZ) \iso \LambdaKEn \iso \LambdaKE\oplus \bZ\cdot \delta$ with
   $q(\delta)=-2$. 
   The line bundle $\Atild\in \Pic(\Xtild)\subseteq H^2(\Xtild,\bZ)$ can formally be decomposed into
  $\Atild=a\lambda +  b\delta$, where 
  $\lambda\in \LambdaKE$ is a primitive element and $a,b\in \bZ_{\geq 0}$. 
  Since by assumption $m=2$ and $(\Xtild, \Atild)\in \cM_{d,m}$, the line bundle $\Atild$ is primitive
  and satisfies $\div(\Atild)=2$.
 This implies that $a=2m$ is
  even (compare Remark \ref{rem:computediv}) and $b$ is odd, i.e.\,$b=2l+1$. Note that since
  $\LambdaKE$ is even, $q(\lambda)=2d_0$ for some integer $d_0$.
  This implies that 
  \begin{equation*}
    q(\Atild)=q(a\lambda+b\delta)= a^2 \cdot 2d_0 -b^2 \cdot 2 = 2\cdot (4m^2 d_0 -(2l+1)^2) 
    = 2\cdot (4\cdot (m^2d_0 - l^2 - l) -1).
  \end{equation*}

  Set $k\coloneqq m^2d_0 - l^2 - l$. Note that $q(\Atild)>0$ implies $k>0$.

  Consider the K3 surface $S$ as given in the statement of the lemma, and fix the notation from page \pageref{notation}.
  Then the primitive line bundle $A\coloneqq H + (2k-1) L=2k\cdot H - (2k-1) \delta$ satisfies:
  \begin{equation*}
    q(A)=q(2k\cdot H - (2k-1) \delta) 
    = (2k)^2 \cdot 2 + (2k -1)^2 \cdot (-2)
    = 2 \cdot (4k-1).
  \end{equation*}

  Therefore, $q(A)=q(\Atild)=2d$. On the other hand Remark \ref{rem:computediv} shows that
  $\div(A)=2$, as claimed.
\end{proof}

\subsection{Proof of Theorem \ref{thm:genbpf4}} \label{sec:proofofthm}
We finally combine the results we collected so far:
\begin{proof}[Proof of Theorem \ref{thm:genbpf4}]
  Fix $(d,m)$ such that $\cM_{d,m} \neq \emptyset$. First of all note, that 
it suffices to find an
  irreducible symplectic variety $X$ of K3$^{[n]}$-type with a primitive and base point free line bundle $A\in
  \cC_X\cap \Pic(X)$ which is not necessarily ample, such that $q(A)=2d$ and $\div(A)=m$: 
  Indeed, a 
  very general deformation $(Y,A_Y)$ of $(X,A)$ satisfies $\rho(Y)=1$. Since $q(A_Y)=q(A)>0$, the projectivity
  criterion (\cite[Theorem 3.11]{Huybrechts:HK:basic-results} and \cite[Theorem 2]{Huybrechts:HK:basic-results:erratum}) implies that $A_Y$ is ample (in fact $A_Y\in \cC_Y$ since $A\in
  \cC_X$ and the family preserves the line bundle; therefore $-A_Y$ is not ample).  Then for a generic deformation
  $(\Xtild,\Atild)$ of $(X,A)$ and $(Y,A_Y)$, the line bundle $\Atild$ will be base point free and ample
  (since both properties are open in families). Since furthermore $\Atild$ is primitive, $q(\Atild)=q(A)=2d$ and
  $\div(\Atild)=\div(A)=m$, 
  the pair $(\Xtild,\Atild)$ defines an element in $\cM_{d,m}$. 
  By Proposition  \ref{prop:candeform} the moduli space $\cM_{d,m}$ is connected. Therefore, the existence of
  $(\Xtild,\Atild)$ implies that for a generic element
  $(X_t,A_t)\in \cM_{d,m}$ the line bundle $A_t$ is base point free (again by  openness of base point
  freeness in families). 
  
  Recall that since $\cM_{d,m}$ parametrizes pairs with primitive line bundles, and the discriminant of
  $\LambdaKEtwo$ is two, the assumption that $\cM_{d,m}\neq \emptyset$ implies that the divisibility
  $m$ can only be one or two (compare Remark \ref{rem:div-discr}).
  
  Suppose $m=1$. Then Lemma \ref{lem:exnice1} gives a K3 surface $S$ with $\Pic(S)\iso \bZ\cdot
  A_S$ for an ample line bundle $A_S\in \Pic(S)$, and shows that the associated line bundle $A\in
  \Hilb^2(S)$ is primitive and
  satisfies that $q(A)=2d$ and $\div(A)=1$. Lemma \ref{lem:Hbpfinexample} shows that the line bundle $A$ is
  base point free. This completes the proof in the case $m=1$.

  Suppose $m=2$ and still assume $\cM_{d,m}\neq \emptyset$. Then Lemma \ref{lem:exnice2} shows that with the
  notation of page \pageref{notation} the primitive line bundle $A\coloneqq H+(2k-1)L$ satisfies
  $2\cdot(4k-1)=q(A)=2d>0$, and $\div(A)=2$. Therefore, the assumption $(d,m)\neq(3,2)$ implies that $k\geq 2$. 
  However for $k\geq 2$, the associated line bundle $A'=H'+(2k-1)L'\in \Pic(X')$ is nef by Lemma
  \ref{lem:cones}.\ref{it:cones-d}. Then $A'$ is base point free by Proposition \ref{prop:okinspeccase},
  which concludes the proof.
\end{proof}

\begin{remark}
  In this proof, we needed to assume that $(d,m)\neq (3,2)$, because we do not know any primitive base point free
  line bundle with these numerical properties so far. 
  
  Note that with the notation of page \pageref{notation}, $H+L \in \Pic(X)$ is a nef line
  bundle with these constants, but Proposition \ref{prop:mostlyokonX} does not imply base point freeness
  of $H+L$. 
  However, Proposition \ref{prop:mostlyokonX} shows that $2(H+L)$ is base point free. Therefore, 
  even in the case $(d,m)=(3,2)$ one can deduce that for a generic pair $(X,A)\in
  \cM_{3,2}$ the line bundle $2A$ is base point free. 

  This will also follow from Theorem \ref{thm:genbpf!}.
\end{remark}

\section{The geometry of a particular example}
\label{sec:geometryEX}
This section is devoted to a detailed study of the geometry of the central example in Section
\ref{sec:BPinexample}. We will use these results in Section \ref{sec:2dimBL} to determine the base locus of
the line bundle $H+L$ from page \pageref{notation}.

Keep the notation as on page \pageref{notation}. In particular $S$ is a K3 surface with $\Pic(S)=
\bZ\cdot H_S$, where $H_S$ is an ample line bundle with $(H_S)^2=2$. 
Such a K3 surface is endowed in a natural way with a map 
\begin{equation*}
  \pi\colon S\overset{2:1}{\longrightarrow} \bP^2, 
\end{equation*}
which is the covering associated to the complete linear system $|H_S|$. Remember that $H_S$ is base point
free by Corollary \ref{cor:bpfifno0-class}. It is well-known
that $\pi$ is a (2:1)-cover branched along a sextic curve $B\subseteq \bP^2$ (see e.g.~\cite[Remark II.2.4.1]{HuybrechtsK3}).

This induces an involution $\rho \colon S\to S$, which interchanges the two sheets of $\pi$.
Consider the graph 
\begin{equation*}
\Gamma_\rho \subseteq S\times S.
\end{equation*}
Denote the map, which interchanges the two factors of $S\times S$ by
\begin{equation*}
\tau\colon S\times S \to S\times S,   \quad (x,x') \mapsto (x',x).
\end{equation*}
Note that $\Gamma_\rho$ is $\tau$-invariant (since $\rho$ is an involution).

Let $\widehat{S\times S}$ be the blow-up of $S\times S$ along the diagonal $\Delta_S\subseteq S\times S$.
Then $\tau$ induces a map
  $\widehat{\tau}\colon \widehat{S\times S}\to \widehat{S\times S}$.
Recall that $X=\Hilb^2(S)$ is obtained from $\widehat{S\times S}$ by dividing
out $\widehat{\tau}$.

The strict transform $S\iso \widehat{\Gamma_\rho}\subseteq \widehat{S\times S}$ of $\Gamma_\rho$ is
again $\widehat{\tau}$-invariant, and the action of $\widehat{\tau}|_{\widehat{\Gamma_\rho}}$
corresponds to the usual action of $\rho$ on $S$.
Consequently the image $P$ of $\widehat{\Gamma_\rho}$ in the Hilbert scheme $\Hilb^2(S)$ is isomorphic
to $\bP^2$.

\begin{equation*}
  \begin{tikzcd}[row sep = small]
    S\ar{d}{\iso} &S \ar{d}{\iso}\ar{l}[swap]{id} \ar{r}{\pi}[swap]{(2:1)}& \bP^2 \ar{d}{\iso} &&\\
    \Gamma_\rho \ar[hook]{dd}[swap]{\iota_\Gamma}& \widehat{\Gamma_\rho} \ar[hook]{dd}{\iota_{\hat{\Gamma}}} \ar{l}[swap]{\iso} \ar{r}&P \ar[hook]{dd}{\iota_P}
       &\Xhat  \ar[end anchor=north east]{ldd}[swap]{\phi} \ar{rdd}{\phi'}& \\
    &&&&\\
    S\times S \ar{dd}[swap]{\pi\times \pi}& \widehat{S\times S} \ar[dashed]{dd}[swap]{\pipihat}\ar{l}[swap]{\zeta_S}\ar{r}{\varepsilon_S}& \Hilb^2(S)=X  \ar[dashed]{dd}{\piz}\ar[dashed]{rr}{f}&& X'\ar{dd}{\psi}\\
    &&&&\\
    \bP^2 \times \bP^2 
    & \widehat{\bP^2 \times \bP^2 } \ar{l}[swap]{\zeta_{\bP^2}}\ar{r}{\varepsilon_{\bP^2}}
    & \Hilb^2(\bP^2) \ar{r}{a}& \bP^2\dual \ar{r}{\iso} & \bP(H^0(X',L')\dual) 
  \end{tikzcd}
\end{equation*}

On the other hand let $\widehat{\bP^2 \times \bP^2 }$ be the blow-up of $\bP^2 \times \bP^2$ along the
diagonal $\Delta_{\bP^2}$. Let $\tau_{\bP^2}$ be the map interchanging the factors of $\bP^2 \times
\bP^2$, and $\widehat{\tau_{\bP^2}}$ be the induced map on $\widehat{\bP^2 \times \bP^2 }$. Then again
$\Hilb^2(\bP^2)$ is obtained from $\widehat{\bP^2 \times \bP^2 }$ by dividing out
$\widehat{\tau_{\bP^2}}$.

There is a natural map $a\colon \Hilb^2(\bP^2) \to \bP^2\dual$, that associates to a subscheme the line
which it spans. 

Recall that from Lemma \ref{lem:cones} we know that there exists a unique other birational model of
$f\colon X\dashrightarrow X'$
which is an irreducible symplectic variety. The associated line bundle $L'=H'-\delta'$ is known to induce
a Lagrangian fibration  $\psi\colon X'\to \bP(H^0(X',L')\dual)$ by Theorem \ref{thm:Matsushita}. 

We will now show that $f$ is exactly the Mukai flop in $P\iso \bP^2 \subseteq X$ and that there exists a
natural identification $\bP^2\dual \iso \bP(H^0(X',L')\dual)$ such that $a\circ \piz$
and $\psi \circ f$ coincide outside of $P$.

\begin{lemma}
  The indeterminacy locus of $\piz$ is exactly $P$.
\end{lemma}
\begin{proof}
  The map $\pipihat$ has indeterminacy locus $\widehat{\Gamma_\rho}$ due to the fact that
  $\widehat{\Gamma_\rho}$ is mapped to the diagonal $\Delta_{\Pz}$ via $(\pipi)\circ \zeta_S$, which is
  blown up in $\PPhat$. Since $\piz$ is the induced map after dividing out by the  actions
  $\widehat{\tau_{\bP^2}}$ and $\widehat{\tau_{S}}$, which correspond to interchanging the two factors of
  the products on both sides, the indeterminacy locus of $\piz$ is exactly
  $\eps_S(\widehat{\Gamma_\rho}) =P$.
\end{proof}

With the notation of page \pageref{notation} (in particular $H\in \Pic(X)$ is the associated line bundle
to $H_S\in\Pic(S)$), we get:
\begin{lemma}\label{lem:Lispullback}
  The pullback $(a\circ \piz)^*(\dO(1))$ is exactly $L=H-\delta \in \Pic(X)$.
\end{lemma}
\begin{proof}
Since $ \eps_S^*$ is injective, one can  compare $(a\circ \piz\circ
\eps_S)^*(\dO(1))$ and $\eps_S^*(L)$ instead. Note that $\eps_S^*(L)=\eps_S^*(H)-\eps_S^*(\delta) =
\zeta_S^*(H_S\boxtimes H_S)- E_S$, where $E_S$ is the exceptional divisor of the blow-up $\SShat$ (use
\eqref{eq:epsalpha} and \eqref{eq:epsdelta}).

One can check that $(a\circ\eps_{\bP^2})^*(\dO(1))=\zeta_{\bP^2}^*\dO(1,1)-E_{\bP^2}$, where $E_{\bP^2}$ is the exceptional
divisor in the blow-up $\PPhat$ (indeed, since $a \circ \eps_{\bP^2}$ is a flat map, it is sufficient to
check that the preimage of the line
$\{x_0\dual=0\}\subseteq \Pdual$ is exactly the strict transform of $\{x_1\otimes x_2- x_2\otimes
x_1=0\}\subseteq \PxP$, and this is a zero set of the line bundle $\zeta_{\bP^2}^*\dO(1,1)-E_{\bP^2}$).

This implies that $(a\circ\eps_{\bP^2}\circ \pipihat)^*(\dO(1))=\zeta_{S}^*(H_S\boxtimes H_S)-E_{S}$, since
$\zeta_{\bP^2}\circ\pipihat$ is well-defined outside $\widehat{\Gamma_\rho}$ and coincides with
$\zeta_S\circ (\pipi)$ on its domain of definition. Furthermore, note
that $\eps_{\bP^2} \circ \pipihat$ and $\piz\circ\eps_S$ are both well-defined outside
$\widehat{\Gamma_\rho}$ and coincide on $\SShat\setminus \widehat{\Gamma_\rho}$, which shows that 
\begin{equation*}
  (a\circ\piz\circ\eps_S)^*(\dO(1))=(a\circ\eps_{\bP^2}\circ \pipihat)^*(\dO(1))
  =\zeta_{S}^*(H\boxtimes H)-E_{S} = \eps_S^*(L).
\end{equation*}
This completes the proof.
\end{proof}

\begin{remark}\label{rem:natidentification}
  In particular this implies that there is a natural map 
  \begin{equation*}
    H^0(\bP^2\dual,\dO(1)) \iso H^0(\Hilb^2(\bP^2), a^*\dO(1)) \inj H^0(X,L) \iso H^0(X',L').
  \end{equation*}
Since it is known that $h^0(X',L')=3$ (compare Theorem \ref{thm:Matsushita}), this is in fact an
isomorphism. 
Therefore, it gives rise to a natural identification $\bP^2\dual \iso \bP(H^0(X',L')\dual)$.
\end{remark}

In the following we will show:
\begin{proposition}\label{prop:Cextremalray}
  A curve $D\subseteq X$ lies in $P$ if and only if $\deg(L|_D)< 0$.
  
  In particular this implies that $X'$ is the Mukai flop of $X$ in $P$.
\end{proposition}

We split  the proof into several lemmas.
\begin{lemma}
  Let $D\subseteq X$ be an irreducible curve which is not contained in $P$. 
Then $\deg(L|_D)\geq 0$.
\end{lemma}
\begin{proof}
  Since $D$ is not contained in the indeterminacy locus of $\pi^{[2]}$ (which is $P$), this follows from the
  fact that $L=\pi^{[2]*}\circ a^* \dO(1)$, since $a^*\dO(1)$ is nef.
\end{proof}

\begin{lemma}\label{lem:C.H}
  Let $C\subseteq P\iso \bP^2$ be a line. Then $\deg(H|_C)=2$.
\end{lemma}
\begin{proof}
  The degree of $H$ on $C$ is:
\begin{align*}
  \deg&(H|_C)=\deg({\iota_P}_*\dO(1).H)
  =\deg({\iota_P}_*\half\pi_*\pi^*\dO(1).H)
  =\half\deg({\varepsilon_S}_*{\iota_{\hat{\Gamma}}}_*\pi^*\dO(1).H)\\
  &=\half\deg({\iota_{\hat{\Gamma}}}_*\pi^*\dO(1).\,{\varepsilon_S}^*H)
  \overset{(*)}{=}\half\deg({\iota_{\hat{\Gamma}}}_*H_S.\,{\zeta_S}^*(H_S\boxtimes H_S))\\
  &=\half\deg({\zeta_S}_*{\iota_{\hat{\Gamma}}}_*H_S.\,H_S\boxtimes H_S)
  =\half\deg({\iota_{\Gamma}}_*H_S.\,H_S\boxtimes H_S)
\\
  &=\half\Big(\deg(H_S. H_S) + \deg(H_S. \rho^*(H_S))\Big)
  =\half\cdot 2 \deg(H_S.H_S)= 2,
\end{align*}
where $(*)$ follows from \eqref{eq:epsalpha}.
This is what we claimed.
\end{proof}

We can deduce the followig:
\begin{lemma}\label{lem:miishalf}
  The constant $m\coloneqq m_i$ associated to $P \subseteq X$ (as introduced in Lemma \ref{lem:degphiA}) is 
  $m=\half$. 
  In particular for a line $C\subseteq P\iso \bP^2$ and for all line bundles $A\in \Pic(X)$ one has $\deg(A|_C)= \half (A,W)_q$.
\end{lemma}
\begin{proof}
Let $C\subseteq P\iso \bP^2$ be a line in $P$.
From the proof of Lemma \ref{lem:degphiA} one can see that $m\in \bQ_+$ is the factor such that 
$\deg(A|_C)= m\cdot (A,W)_q$ for all line bundles $A\in \Pic(X)$.
The degree of $H$ on $C$ is $\deg(H|_C)=2$ by Lemma \ref{lem:C.H}.
Therefore, the factor $m$ is
  exactly $\frac{1}{2}$, since $(H,W)_q=4$. 
\end{proof}
\begin{remark}
 In particular Proposition \ref{prop:mostlyokonX} does not
  show base point freeness for the line bundle $H+L$ in this situation (compare Remark
  \ref{rem:halfisproblem}).  In fact, we will show in the next section that in our
  situation $H+L$ has a non-trivial base locus.
\end{remark}

\begin{corollary}\label{cor:DisLnegative}
  The line $C\subseteq P\iso \bP^2$ satisfies $\deg(L|_C)=-1$, and therefore any curve $D\subseteq P$ has
  $\deg(L|_D)<0$. 
\end{corollary}
\begin{proof}
  This follows immediately from Lemma \ref{lem:miishalf}, since $\deg(L|_C)= \half (L,W)_q=-1$.
\end{proof}

\begin{proof}[Proof of Proposition \ref{prop:Cextremalray}]
  Note that Lemma \ref{lem:C.H} and Corollary \ref{cor:DisLnegative} imply that  an irreducible curve
  $D\subseteq X$ 
  lies in $P$ if and only if $\deg(L|_D)< 0$. 
  
  Therefore, $C$ defines an extremal ray in $\NEbar(X)$, which is negative with respect to the
  pair $(X,\nu L)$. For $0<\nu \ll 1$ the pair $(X,\nu L)$ is klt (this is true for every effective
  divisor on $X$). Furthermore, the union of all curves in this extremal ray is $P$.
  Then \cite[Proposition 2.1]{Wierzba} implies that $X'$ exists and that it is exactly the Mukai flop in
  $P$.
\end{proof}

\begin{corollary}\label{cor:mapscoincideoutsideP}
  The rational maps $a\circ \piz$ and $\psi \circ f$ coincide outside of
  $P$ (up to the natural identification $\bP^2\dual \iso \bP(H^0(X',L')\dual)$ from Remark
  \ref{rem:natidentification}).
\end{corollary}
\begin{proof} 
Both rational maps are induced by the global sections of 
\begin{equation*}
  (a\circ \piz)^*(\dO(1))=L=f^*(L')=(\psi \circ f)^*(\dO(1)), 
\end{equation*}
  and the base locus of $L$ is contained in $P$, since $L'$ is base point free, and $f$ is the Mukai flop
  in $P$.
\end{proof}

\section{Two-dimensional base locus of a line bundle with divisibility two}
\label{sec:2dimBL}
Let again $S$ be a K3 surface with $\Pic(S)=\bZ\cdot H_S$ for an ample line bundle with $(H_S)^2=2$, and
keep the notation from page \pageref{notation}. 
In this section, we use the  explicit geometry of $X=\Hilb^2(S)$ in order to study the base
locus of the line bundle $H+L$. The main result is:
\begin{theorem}\label{thm:baselocusH+L}
  The ample line bundle $H+L\in \Pic(X)$ has a non-trivial base locus which is isomorphic to $\bP^2$.
\end{theorem}

For the proof, we consider the map
\begin{equation*}
  \mu \colon H^0(X,L) \otimes H^0(X,H) \to H^0 (X,H+L),
\end{equation*}
and we show:
\begin{proposition}\label{prop:musurj}
  The map $\mu$ is surjective. 
\end{proposition}

This immediately implies Theorem \ref{thm:baselocusH+L}:
\begin{proof}[Proof of implication: Proposition \ref{prop:musurj} $\implies$ Theorem \ref{thm:baselocusH+L}]
  Consider a line $C\subseteq P$. Use Corollary \ref{cor:DisLnegative} to see that the degree
  $\deg(L|_C)=-1$ is negative.
Therefore, $H^0(C,L|_C)=0$, and thus every global section of $L$ has base points along $P$. 
Since $\mu$ is surjective, this implies that every global section of $H+L$ vanishes along $P$ (since it
has a factor coming from $H^0(X,L)$ which already vanishes along $P$).
\end{proof}
For the rest of this section we prove Proposition \ref{prop:musurj} in several steps.

\begin{lemma} \label{lem:dimcount}
The spaces of global sections which occur have the following dimensions:
  \begin{compactenum}
      \item $h^0(X,L) = 3$, \label{part:h0L}
      \item $h^0(X,H) =6$, and \label{part:h0H}
      \item $h^0(X,H+L)= 15$. \label{part:h0H+L}
  \end{compactenum}
\end{lemma}
\begin{proof}
  Part \ref{part:h0L} holds, because $L'$ is the primitive line bundle which induces the Lagrangian
  fibration on $X'$. 
  Then $h^0(X,L)=h^0(X',L') = h^0(\bP^2, \dO(1)) =  3$.

  Since $H$ is big and nef, use Kodaira vanishing and the Riemann--Roch formula for K3$^{[2]}$-type (see
  Proposition \ref{prop:EGL})
  to observe part \ref{part:h0H}:
  \begin{equation*}
    h^0(X,H)=\chi(X,H)=\binom{\frac{1}{2}q(H)+3}{2}=\binom{4}{2}=6.
  \end{equation*}
  
  The same arguments apply for part \ref{part:h0H+L}, which shows
   $ h^0(X,H+L)=\binom{6}{2}=15$, 
  since $q(H+L)=6$.
\end{proof}

In the following we determine the dimension of $\ker(\mu)$ in order to prove Proposition \ref{prop:musurj}.

Fix coordinates for $\bP^2$ and the induced coordinates on $\bP^2\dual$. 

\begin{name}[Index convention]
  We will frequently deal with indices in the set $\{0,1,2\}$. In order to keep the notation as slender
    as possible, we will identify this set  with $\bZ/3\bZ$. Let us further write $I\coloneqq \bZ/3\bZ$.
\end{name}

Since $\pi$ is induced by $|H_S|$ the basis $\{x_i\}_{i\in I}\subseteq H^0(\Pz,\dOi)$ gives rise to a basis
$\{\pi^*(x_i)\}_{i\in I}\subseteq H^0(S, H_S)$.

Note that for every line bundle $A\in \Pic(X)$ the map
\begin{equation}\label{eq:injective}
\eps_S^*\colon H^0(X,A) \inj H^0(\SShat, \eps_S^*A)
\end{equation}
is injective.

The global sections $H^0(X,H)$ are given by symmetric polynomials in the global sections $H^0(S,H_S)$ in
the following sense: 
A basis $\{t_i^X,v_i^X\}_{i\in I}\subseteq H^0(X, H)$ is determined by the property 
\begin{align}
  &\eps_S^*(t_i^X)= \zeta_S^*(\pipi)^*(x_{i+1}\otimes x_{i+2} + x_{i+2}\otimes x_{i+1} ) \quad {\rm
    and }\notag \\
  &\eps_S^*(v_i^X)= \zeta_S^*(\pipi)^*(x_i\otimes x_i) . \label{eq:vi}
\end{align}
(One way to see that these span already the global sections of $H$ is Lemma \ref{lem:dimcount}.\ref{part:h0H}).

For $i\in I$ define the following sections in $H^0(\PxP, \dOii)$:
\begin{align*}
  &t_i= x_{i+1}\otimes x_{i+2} + x_{i+2}\otimes x_{i+1} \quad {\rm and }\\
  &v_i= x_i\otimes x_i ,
\end{align*}
and let $W\subseteq H^0(\PxP, \dOii)$ be the six-dimensional  subspace spanned by the $\{t_i\}_{i\in I}$
and $\{v_i\}_{i\in I}$.

We now want to describe a basis in $H^0(X,L)$. 
Consider the basis $\{x_i\dual\}_{i\in I}\subset H^0(\Pdual,\dOi)$, which determines a basis
$\{s_i^X\coloneqq f^*\psi^*(x_i\dual)\}_{i\in I} \subseteq H^0(X,L)$ via pullback. By the injectivity of
\eqref{eq:injective}, it is enough to describe the elements $\{\eps_S^*(s_i)\}\subseteq
H^0(\SShat,\eps_S^*(L))$.
Note that 
\begin{equation*}
  \eps_S^*(s_i^X) = \eps_S^*(f^*\psi^*(x_i\dual))
  =\pipihat^* (a \circ \eps_{\Pz})^* (x_i\dual),
\end{equation*}
since  $a \circ \eps_{\Pz}\circ \pipihat$ and $\psi\circ f\circ \eps_S$ are both well-defined outside the
two-dimensional set $\widehat{\Gamma_\rho}$ and coincide on their domain of definition by Corollary
\ref{cor:mapscoincideoutsideP}. 
Let $\{x_i\dual=0\}\subseteq \Pdual$ be the line cut out by $x_i\dual$.
One can verify that  the preimage 
$(a \circ \eps_{\Pz})^{-1}(\{x_i\dual =0\})$ is equal to the strict transform of the set 
\begin{equation*}
\{x_{i+1}\otimes x_{i+2} - x_{i+2}\otimes x_{i+1}=0\}\subseteq \PxP.
\end{equation*}

Fix the notation $s_i\coloneqq x_{i+1}\otimes x_{i+2} - x_{i+2}\otimes x_{i+1} \in H^0(\PxP, \dOii)$, and let
$V\subseteq H^0(\PxP, \dOii)$ be the subspace spanned by the $\{s_i\}_{i\in I}$.

Denote by $E_{\Pz}\subseteq \PPhat$ the exceptional divisor of the  blow-up, and pick a section 
$s_E\in H^0(\PPhat, \dO(E_{\Pz}))$.
With this notation the strict transform of the set $\{s_i=0\}$ is cut out by
\begin{equation*}
  \zeta_{\bP^2}^*s_i\otimes s_E^{-1} \in H^0(\PPhat, \zeta_{\Pz}^*\dOii
  \otimes \dO(- E)).
\end{equation*}
Up to renormalizing $s_E$, we therefore get that 
\begin{align}
  \eps_S^*(s_i^X) &= \pipihat^*(a\circ \eps_{\bP^2})^*(x_i\dual)
= \pipihat^*(\zeta_{\Pz}^*s_i\otimes s_E^{-1}) \notag \\ 
  &=\zeta_S^*(\pipi)^*s_i\otimes\pipihat^*( s_E^{-1}). \label{eq:si}
\end{align}

\begin{lemma}\label{lem:comparekernels}
  The kernel of the natural map
  \begin{equation*}
    \mu'\colon V \otimes W \to H^0(\PxP,\dO(2,2))
  \end{equation*}
  has the same dimension as the kernel of $\mu$.
\end{lemma}

\begin{proof}

Set $U\coloneqq \SxS\setminus \Delta_S\iso \SShat \setminus E_S$, where $E_S$ is the exceptional divisor, and consider the composition 
\begin{align*}
   H^0(X,L) \otimes H^0(X,H) \overset{\mu}\too &H^0 (X,H+L)\\
   &\inj H^0(\SShat, \eps_S^*(H+L)) 
   \inj H^0(U, \eps_S^*(H+L)|_U)
\end{align*}
On the other hand there is a composition
\begin{align*}
  V\otimes W \overset{\mu'}{\too} H^0(\PxP,\dO(2,2)) 
  &\inj H^0(\SxS, H_S^{\otimes 2}\boxtimes H_S^{\otimes 2}) \\
  &\inj H^0(\SShat, \eps_S^*(H+L)) 
  \inj H^0(U, \eps_S^*(H+L)|_U).
\end{align*}

The relations \eqref{eq:vi} and \eqref{eq:si} imply that the images of both compositions coincide. 
Since $\dim V= \dim H^0(X,L)$ and $\dim W = \dim H^0(X,H)$ (because they have corresponding base elements), this
implies that the kernel of $\mu'$ has the same dimension as the kernel of $\mu$.
\end{proof}

\begin{proposition}\label{prop:kermu}
  The kernel of the map $\mu'$ is three-dimensional.
\end{proposition}
\begin{proof}
  The kernel of $\mu'$ is spanned by the three elements
   $ s_{k+1}\otimes t_{k+2} + s_{k+2}\otimes t_{k+1} + 2 s_k \otimes v_k$
  for $k\in I$.

  The fact that these elements lie in the kernel can immediately be verified by inserting the definitions of $s_k,
  t_k$, and $v_k$.
  The other inclusion can be shown by an explicit computation with basis vectors in
  $H^0(\bP^2\times \bP^2,  \dO(2,2))$.
\end{proof}
\begin{proof}[Proof of Proposition \ref{prop:musurj}]
  By Lemma \ref{lem:dimcount} the dimension of $H^0(X,L)\otimes H^0(X,H)$ is 18, and the
  dimension of $H^0(X, H+L)$ is 15.
  
  On the other hand, Proposition \ref{prop:kermu} shows that the kernel of $\mu'$ is
  three-dimen\-sional. By Lemma \ref{lem:comparekernels}, this implies that the kernel of $\mu$ is
  three-dimensional as well. 
  
  Therefore, the image of $\mu$ is 15-dimensional, and thus $\mu$ is surjective.

  Note that $H+L$ is ample, since it is in the interior of the nef cone (compare Lemma
  \ref{lem:cones}.\ref{it:cones-c}). 
\end{proof}

In particular this completes the proof of Theorem \ref{thm:baselocusH+L}.


\section{Generic base point freeness for \texorpdfstring{$H+L$}{H+L}}\label{sec:genbpfforH+L}
Let again $X\coloneqq \Hilb^2(S)$ for a K3 surface $S$ with $\Pic(S)=\bZ\cdot H_S$ for an ample line
bundle with $(H_S)^2=2$. Keep the notation of page \pageref{notation}. In particular $H+L\in \Pic(X)$ is
the ample line bundle on $X$ which has base points along a $\bP^2$ by Theorem \ref{thm:baselocusH+L}.

In this section we show 

\begin{proposition}\label{prop:H+Lgenbpf}
  Let $X$ and $H+L \in \Pic(X)$ be as above (and keep the notation from page \pageref{notation}). 
  Then there exists a family $\cX \to T$ over a one-dimensional connected base $T$ with $\cA\in
  \Pic(\cX)$ such that 
  \begin{compactenum}
    \item there is $0\in T$ such that $\cX_0= X$ and $\cA_0 =H+L$, \label{it:defofXH+L}
    \item for all $t\in T$ the restriction $\cA_t$ is ample, and \label{it:Atample}
    \item for some $t_0\in T$ the line bundle $\cA_{t_0}$ is base point free. \label{it:Aisbpf}
  \end{compactenum}
\end{proposition}

This completes the proof of the following theorem:
\begin{theorem}\label{thm:genbpf!}
  Consider the moduli space $\cM_{d,m}$ of $(d,m)$-polarized irreducible symplectic varieties of
  K3$^{[2]}$-type. Then for a generic pair $(\Xtild, \Atild)\in \cM_{d,m}$ the line bundle $\Atild$ is
  base point free.
\end{theorem}

\begin{proof}[{Proof of the implication Proposition \ref{prop:H+Lgenbpf} $\Rightarrow$ Theorem \ref{thm:genbpf!}}]
  Theorem \ref{thm:genbpf4} shows the statement unless $(d,m)=(3,2)$. Since base point freeness is an
  open property in families (and since $\cM_{d,m}$ is connected by Proposition \ref{prop:candeform}), it
  suffices to find a pair $(X_{t_0},A_{t_0})\in \cM_{3,2}$ with $A_{t_0}$ base point free.
  
  We claim that the pair $(\cX_{t_0},\cA_{t_0})$ from Proposition \ref{prop:H+Lgenbpf} is such a
  pair. First, note that $\cA_{t_0}$ is primitive since 
  it is the deformation of the primitive line bundle $H+L=\cA_0$. Further,
  $q(\cA_{t_0})=q(\cA_0)=q(H+L)=6$ and $\div(\cA_{t_0})=\div(\cA_0)=\div(H+L)=2$. Finally $\cA_{t_0}$ is
  ample by  Proposition \ref{prop:H+Lgenbpf}.\ref{it:Atample}.
\end{proof}

 For the proof of Proposition \ref{prop:H+Lgenbpf},  use the scheme structure on the base locus of an
 effective line bundle.
Recall that the base locus of a line bundle can be equipped with a natural scheme structure in the
following way:
\begin{definition}\label{def:B}
  For a variety $X$ and $A\in \Pic(X)$ with $h^0(X,A)>0$, define $\fB(A)\subseteq X$ as the subscheme associated to the ideal sheaf
   \begin{equation*}
     \fI(A)\coloneqq {\rm im}\big(H^0(X,A)\otimes A^{-1} \overset{\rm ev}{\too} \dO_X\big)\subseteq \dO_X. 
   \end{equation*}
\end{definition}

 The following lemma is easy to observe:
 \begin{lemma}\label{lem:basescheme}
   Let $X$ be a variety and $A,A'\in \Pic(X)$ effective line bundles.
   \begin{enumerate}
   \item The support of $\fB(A)$ is exactly the base locus of $A$.\label{it:suppisBL}
   \item If $A$ is base point free, then $\fI(A)=\dO_X.$\label{it:Bifbpf}
   \item If $A$ is base point free, then $\fB(A+A') \inj \fB(A')$ is a closed immersion.\newline
     If furthermore $H^0(X,A)\otimes H^0(X,A')\to H^0(X,A+A')$ is surjective, then 
     $\fB(A+A') = \fB(A')$\label{it:Bforproducts}. 
   \item If $\phi\colon \Xhat \to X$ is a morphism such that the pullback induces an isomorphism $H^0(X,A)\iso H^0(\Xhat, \phi^*(A))$, then 
     $\fB(A)\times_X \Xhat = \fB(\phi^*A)$. \label{it:Bofpullback}
   \end{enumerate}
 \end{lemma}
 \begin{proof} We leave out the proof, which consists of very elementary arguments.
 \end{proof}

We now return to our example:  $X\coloneqq \Hilb^2(S)$ for a K3
surface $S$ with $\Pic(S)=\bZ\cdot H_S$ for an ample line 
bundle with $(H_S)^2=2$ with the notation of page \pageref{notation}, and Section
\ref{sec:geometryEX}. In particular $X\leftarrow \Xhat \to X'$ is the Mukai flop in $P\subseteq X$.
 For the proof of Proposition \ref{prop:H+Lgenbpf},  we will first study the scheme structure of the base locus of 
$H+L$. 

\begin{proposition}\label{prop:structureofBH+L}
  With the notation of the previous sections, the scheme structure of the base locus of $H+L$ is
  $\fB(H+L)=P$, where $P$ is equipped with its reduced induced scheme structure.
\end{proposition}

\begin{proof}
  By Theorem \ref{thm:baselocusH+L} and Lemma \ref{lem:basescheme}.\ref{it:suppisBL},
  the support of $\fB(H+L)$ is $P$. 
  Furthermore, $H$ is base point free by Lemma \ref{lem:Hbpfinexample}, and Proposition \ref{prop:musurj} shows that $\mu \colon H^0(X,L) \otimes H^0(X,H) \to H^0
  (X,H+L)$ is surjective.
  Therefore, $\fB(H+L)=\fB(L)$ by Lemma \ref{lem:basescheme}.\ref{it:Bforproducts}.
  Let $E\subseteq \Xhat$ be the exceptional divisor of $\phi\colon \Xhat\to X$.

  Suppose for contradiction, that $\fB(L)$ does not have the reduced structure. Then also 
  $\fB(L)\times_X \Xhat$ would be a non-reduced scheme supported on $E$. By Lemma
  \ref{lem:basescheme}.\ref{it:Bofpullback}, we 
  know $\fB(L)\times_X \Xhat = \fB(\phi^*(L))$. Therefore, it suffices to show that $\fB(\phi^*(L))$
  is in fact reduced to get the desired contradiction.

  Corollary \ref{cor:compareLB} and Lemma \ref{lem:miishalf} show that  $\phi^*(L)=
 \phi'^*(L')+E$.
  
  Since $\phi'^*(L')$ is base point free (because $L'$ is base point free), Lemma
  \ref{lem:basescheme}.\ref{it:Bforproducts} implies that $\fB(\phi^*(L))=\fB(\phi'^*(L')+E) \inj\fB(E)$  is a
  closed immersion. Note however, that $\fB(E)=E$ with the reduced induced structure. 
  
  Taking everything together, we showed that $\fB(\phi^*(L))\inj E\subseteq \Xhat$ is a subscheme with
  support $E$ and thus $\fB(\phi^*(L))$ has the reduced structure. As we pointed out earlier, this give
  the desired contradiction.
\end{proof}

\begin{corollary}\label{cor:descrofdeformedBL}
  Let $\xi\colon \cX\to T$ be a family over a one-dimensional connected base, and $\cA\in
  \Pic(\cX)$ 
such that 
 there is $0\in T$ with $\cX_0= X$ and $\cA_0 =H+L$.
 Then there is an open set $U\subseteq T$ such that 
     for all $t\in U$ the restriction $\cA_t$ is ample, and the 
      base loci of the $\cA_t$ are either empty for all $t\in U\setminus\{0\}$, or they form a (flat) family of
      surfaces with special fibre $P$.
\end{corollary}
\begin{proof}
  Since $H+L$ is ample,
  there exists an open subset $U\subseteq T$ such that for all $t\in U$ the restriction $\cA_t$
  is ample.  Kodaira vanishing shows that $h^1(X,\cA_t)=0$ for all $t\in U$.
  Consequently, one can apply  base-change (see e.g.~\cite[Theorem 5.10]{FGA-explained}) to see that 
  $R^0{\xi_U}_*(\cA_U)={\xi_U}_*(\cA_U)$ is locally free
  of rank $h^0(X, H+L)=15$ (compare Lemma \ref{lem:dimcount}.\ref{part:h0H+L}), where $\xi_U$ is the
  restriction of $\xi$ to the preimage $\cX_U$ of $U$, and $\cA_U\coloneqq \cA|_{\cX_U}$. 
  After shrinking $U$, we may therefore assume that ${\xi_U}_*(\cA_U) \iso \dO_U^{\oplus 15}$, and pick a basis
  $s^1,\dots, s^{15} \in 
  H^0(U,{\xi_U}_*(\cA_U))\iso H^0(\cX_U,\cA_U)$ corresponding to the standard sections in
  $H^0(U,\dO_U^{\oplus 15})$.

  Let $\cB\coloneqq \fB(s^1,\dots, s^{15})\subseteq \cX_U$ be the vanishing scheme of the sections
  $s^1,\dots ,s^{15}$ 
  (defined analogously to Definition \ref{def:B}). 
  One can check that 
  \begin{equation*}
    \cB\times_{\cX_U} X =\fB(s^1,\dots, s^{15})\times_{\cX_U} X 
    = \fB(s^1|_X, \dots , s^{15}|_X) = \fB(\cA_0)=\fB(H+L).
  \end{equation*} 
  Since we know that $ \cB\times_{\cX_U} X=\fB(H+L) = P$ with the reduced induced structure by Proposition
  \ref{prop:structureofBH+L},
  $\cB$ has an irreducible component which meets $X$ in $P$ and no other component of $\cB$ can meet $X$
  (since this would cause a non-reduced structure of $\cB\times_{\cX_U} X$).  
  
  By shrinking $U$, we may assume that $\cB$ has only one irreducible component. If $\cB$ maps
  surjectively onto $U$, then it  is flat over $U$ and has constant fibre dimension (since $U$ is
  one-dimensional), 
  and thus 
  $\cB$ is a family of surfaces. Otherwise $\cB=P$.  To conclude the proof, observe that the base locus
  of $\cA_t$ is 
  $\fB(\cA_t)=\cB\times_{\cX_U} \cX_t$ for every $t\in U$. 
\end{proof}

\begin{proof}[{Proof of Proposition \ref{prop:H+Lgenbpf}}]
  Pick a deformation $(\cX\overset{\xi}{\too} T,\cA)$ of $(X, H+L)$  over a one-dimensional base $T$, such that the
  very general fibre has Picard rank $\rho(\cX_t)=1$.  
  
  We will show that, up to shrinking $T$, this family satisfies all claimed properties. 
  By assumption $(\cX,\cA)$ is a deformation of $(X,H+L)$ and thus there exists $0\in T$ with $\cX_0=X$
  and $\cA_0=H+L$ (which was Part \ref{it:defofXH+L}).

  Corollary \ref{cor:descrofdeformedBL} shows that (up to replacing $T$ by an open $U\subseteq T$), we
  may assume that $\cA_t$ is ample for all $t$ (and thus \ref{it:Atample}), and furthermore
  that one of the following cases will occur: 
  either $\cA_t$ is base point free for all $t\neq 0$, or the base loci of the $A_t$ vary in a flat
  family $\cB$ of surfaces with special fibre $P$.

  In the first case, we are done, since any $t\neq 0$ is base point free, and thus satisfies condition
  \ref{it:Aisbpf}. 

  In the following we will show that the assumption $\rho(\cX_t)=1$ for very general $t\in T$ excludes the
  second case.
  
  Suppose for contradiction that there exists a flat family of surfaces $\cB\subseteq \cX\to T$ with
  special fibre $P$. 
  Since $H+L$ is ample, one can pick $k\in \bN$ such that $k(H+N)$ is very ample. In
  particular there is a section $s_0\in H^0(X,k(H+N))$ which restricts non-trivially to $P$. Let
  $D\subseteq P$ be the intersection of $P$ with the zero-locus of $s_0$. 
  
  As in the proof of Corollary \ref{cor:descrofdeformedBL}, one can see that (up to shrinking $T$) $s_0$
  is the restriction of a section $s\in H^0(\cX,k\cA)$.

  Let $\cD$ denote the intersection of $\cB$ with the zero locus of $s$.
  After potentially shrinking $T$, we can assume that 
   $\cD$ is a flat family of curves. 
   Therefore, the class  $[\cD_t]\in H^6(\cX_t, \bZ)\iso H^6(X,\bZ)$ does not depend on $t$.
   
   For the rest of the proof, we will show that this curve class does not deform to the very general fibre, which gives the
   desired contradiction to the existence of $\cB$.

   Let $C\subseteq P\iso \bP^2$ be a line. There exists $l\in \bN$ with
   $[\cD_0]=l[C]\in H^6(X,\bZ)$.
   
   Lemma \ref{lem:miishalf} shows that $\deg(\alpha.[C])= \half (\alpha,W)_q$ for all 
   $\alpha \in H^{1,1}(X,\bZ)$, where $W= 2H-3\delta$ (with the notation on page
   \pageref{notation}). By Remark \ref{rem:type-argument} both $\deg(\_.[C])$ and $(\_,W)_q$  vanish on the transcendental
   part $H^2(X,\bZ)_{\tr}$ and this implies that 
   $\deg(\alpha.[C])= \half (\alpha,W)_q$ for all $\alpha \in H^2(X,\bZ)$.

   Pick a very general $t\in T$ with $\rho(\cX_t)=1$.   
   Note that $\cA_t \in H^{1,1}(\cX_t,\bZ) \subseteq H^2(\cX_t,\bZ)$ spans $H^{1,1}(\cX_t,\bZ)$ and corresponds to
   $\cA_0=H+L \in H^2(X, \bZ)$. 
   Let $W_t\in H^2(\cX_t,\bZ)$ be the class corresponding to $W\in H^2(X,\bZ)\iso H^2(\cX_t,\bZ)$. Since
   $W$ is not a multiple of $\cA_0$, the class $W_t$ in not contained in
   $H^{1,1}(\cX_t,\bZ)$. Therefore, there exists an element $\alpha_t \in H^2(\cX_t,\bZ)_{\rm tr}$ such
   that $(\alpha_t, W_t)\neq 0$.
   Consequently 
   \begin{equation*}
     \deg(\alpha_t.[\cD_t])=\deg(\alpha_0.[\cD_0])= \deg(\alpha_0.l[C])=l (\alpha_0, W)_q =
     l (\alpha_t, W_t) \neq 0,
   \end{equation*}
   which shows that $[\cD_t]$ is not the class of a curve. 
    
   This is a contradiction and shows that the second case cannot occur, which concludes the proof.
\end{proof}

\bibliographystyle{alpha}
\bibliography{\folder Literatur}

\end{document}